\documentclass[12pt,amstex]{amsart}

\usepackage{mathptmx}
\usepackage{mathrsfs}

\usepackage{verbatim}
\usepackage{url}
\usepackage[all]{xy}
\usepackage{color}

\usepackage[colorlinks=true,citecolor=blue]{hyperref}

\usepackage{stmaryrd}
\usepackage{epsfig}
\usepackage{amsmath}
\usepackage{amssymb}

\usepackage{amscd}
\usepackage{graphicx}
\usepackage{pstricks}
\usepackage{tocvsec2}

\theoremstyle{plain}
\newtheorem{thm}{Theorem}[section]
\newtheorem{lem}[thm]{Lemma}

\newtheorem{prop}[thm]{Proposition}
\newtheorem{ques}{Question}

\newtheorem{cor}[thm]{Corollary}
\theoremstyle{definition}
\newtheorem{de}[thm]{Definition}
\newtheorem{exam}[thm]{Example}
\newtheorem{rem}[thm]{Remark}

\numberwithin{equation}{section}

\def \N {\mathbb N}

\def \Z {\mathbb Z}
\def \R {\mathbb R}
\def \bbs {{\mathbb S}}

\def \O {\mathcal{O}}

\def \F {\mathcal F}
\def \G {\mathcal{G}}

\def \E {\mathcal E}

\def \X {\mathcal{X}}

\def \M {{\bf M}}

\def \id {{\rm id}}

\def \a {\alpha }
\def \b {\beta}
\def \ep {\epsilon}
\def \d {\delta}
\def \D {\Delta}

\def \c {\circ}
\def \w {\omega}

\def \T {\mathbb T}

\begin{document}

\title{Almost proximal extensions of minimal flows}

\author{Yang Cao}

\author{Song Shao}

\address{School of Mathematical Sciences, University of Science and Technology of China, Hefei, Anhui, 230026, P.R. China}

\email{cy412@mail.ustc.edu.cn}

\email{songshao@ustc.edu.cn}



\thanks{This research is supported by NNSF of China (11971455, 11571335).}


\begin{abstract}
In this paper we study almost proximal extensions of minimal flows. Let $\pi: (X,T)\rightarrow (Y,T)$ be an extension of minimal flows. $\pi$ is called an almost proximal extension if there is some $N\in \N$ such that the cardinality of any almost periodic subset in each fiber is not greater than $N$. When $N=1$, $\pi$ is proximal. We will give the structure of $\pi$ and give a  dichotomy theorem: any almost proximal extension of minimal flows is either almost finite to one, or almost all fibers contain an uncountable strongly scrambled subset. Using category method Glasner and Weiss showed the existence of proximal but not almost one to one extensions \cite{GW79}. In this paper, we will give explicit such examples, and also examples of almost proximal but not almost finite to one extensions.
\end{abstract}

\maketitle





\section{Introduction}

The structure theory of minimal flows originated
in Furstenberg's seminal work \cite{F63} for distal minimal flows, and the structure theorem for the general minimal flows was built in Ellis-Glasner-Shapiro \cite{EGS},
McMahon \cite{Mc76}, Veech \cite{V77}, and
Glasner \cite{G76}. Roughly speaking, the class of minimal flows
is the smallest class of flows containing the trivial flow
and closed under homomorphisms, inverse limits, and has three
``building blocks'' which are equicontinuous extensions, proximal extensions and
topologically weakly mixing extensions. In this paper, we mainly study proximal extensions
and almost proximal extensions. Refer to \cite{G76} for a systematical study of proximal flows.

\medskip

In this paper, a {\em flow} $(X,T)$ is a compact metric space $X$ with an infinite countable discrete group $T$ acting continuously on $X$.
Let $\pi: (X,T)\rightarrow (Y,T)$ be an extension of minimal flows. The proximal relation $P(X)$ is defined by $$P(X)=\{(x,x')\in X^2: \inf_{t\in T} d(tx,tx')=0\}.$$
Any pair in $P(X)$ is called a {\em proximal pair}. $\pi$ is {\em proximal} if for any pair $(x,x')$ in the same fiber is proximal, i.e., $(x,x')\in P(X)$ whenever $\pi(x)=\pi(x')$.
A subset $A$ is called an {\em almost periodic set} if every finite subset $\{x_1,x_2,\ldots, x_n\}$ of $A$, $(x_1,x_2,\ldots, x_n)$ is minimal in the product flow $(X^n,T)$.
Note that any pair $(x,x')$ with $x\neq x'$ will not be proximal if $(x,x')$ is minimal. By this fact it is easy to show that $\pi$ is proximal if and only if each almost periodic subset in the fiber of $\pi$ is a singleton. Inspired by this, we call $\pi$ an {\em almost proximal extension} if there is some $N\in \N$ such that the cardinality of any almost periodic subset in each fiber is not greater than $N$. When $N=1$, $\pi$ is proximal.

\medskip

An extension $\pi: (X,T)\rightarrow (Y,T)$ of minimal flows is {\em almost finite to one} if some fiber is finite, i.e., there is some $y$, $\pi^{-1}(y)$ is finite. It is not difficult is see that any almost finite to one extension is almost proximal (Proposition \ref{almost N-1}). After we study the structure of almost proximal extensions, we show that an extension $\pi: X\rightarrow Y$ of minimal flows is almost finite to one if and only if it is almost proximal and point distal.

\medskip

It is an open question \cite{AGHSY} that: is any non point distal minimal flow ( i.e., for any point $x\in X$ , there is $x'\neq x$ such that $(x,x')$ is proximal) chaotic in the sense of Li-Yorke? It was showed in \cite{AGHSY} that if a minimal flow $(X,T)$ is a proximal but not almost one to one extension of some flow $(Y,T)$, then $(X,T)$ is not point distal and it is Li-Yorke chaotic. In this paper, we generalize this result, and show that almost proximal extension has the following dichotomy theorem: any almost proximal extension of minimal flows is either almost finite to one, or almost all fibers contain an uncountable strongly scrambled subset. In particular, if a minimal flow $(X,T)$ is an almost proximal but not almost finite to one extension of some flow $(Y,T)$, then $(X,T)$ is not point distal and it is Li-Yorke chaotic.

\medskip

Since there is no non-trivial proximal minimal flow under abelian group action \cite[Theorem 3.4]{G76}, it is not easy to give a minimal $\Z$-flow which is proximal but not almost one to one extension of its maximal equicontinuous factor. In fact this was a question by Furstenberg several years ago. Using category method Glasner and Weiss showed the existence of proximal but not almost one to one extensions \cite{GW79}. In this paper, using methods in \cite{CM06}, we will give explicit such examples, and also examples of almost proximal but not almost finite to one extensions.
In addition, all examples constructed are uniformly rigid.

\subsection{The organization of the paper}\
\medskip

We organize the paper as follows. In Section \ref{section-tds}, we introduce some basic notions and results needed in the paper. In Section \ref{section-ap}, we introduce the notion of almost proximal and give its structure. In Section \ref{section-wm-ap}, we study chaotic properties of proximal but not almost one to one extensions. In Section \ref{section-main}, We will give a  dichotomy theorem: any almost proximal extension of minimal flows is either almost finite to one, or almost all fibers contain strongly scrambled subset. In Section \ref{section-example}, we will give explicit examples of almost proximal but not almost finite to one extensions. And in the final section, we will give some questions.

\bigskip

\noindent{\bf Acknowledgments:} We would like to thank Professor Wen Huang and Professor Xiangdong Ye for their very useful comments.


\section{Basic facts about abstract topological dynamics}\label{section-tds}

In this section we recall some basic definitions and results in
abstract topological flows. For more details, see \cite{Au88, Ellis,
G76, V77, Vr}. In the article, integers, nonnegative integers and natural numbers
are denoted by $\Z$, $\Z_+$ and $\N$ respectively.

\subsection{Topological transformation groups}\
\medskip

A {\em flow} or a {\em topological dynamical system} is a triple
$\X=(X, T, \Pi)$, where $X$ is a compact Hausdorff space, $T$ is a
Hausdorff topological group and $\Pi: T\times X\rightarrow X$ is a
continuous map such that $\Pi(e,x)=x$ and
$\Pi(s,\Pi(t,x))=\Pi(st,x)$, where $e$ is the unit of $T$, $s,t\in T$ and $x\in X$. We shall fix $T$ and suppress the
action symbol.

In this paper, we always assume that $T$ is infinite countable and discrete, unless we state it explicitly in some places.
Moreover, we always assume that $X$ is a compact metric space with metric $d(\cdot, \cdot)$.

\medskip

When $T=\Z$, $(X,T)$ is determined by a homeomorphism $f$, i.e., $f$ is the transformation corresponding to $1$ of $\Z$. In this case, we usually denote $(X,\Z)$ by $(X,f)$, and also call it {\em discrete flow}.

\medskip

Let $(X,T)$ be a flow and $x\in X$, then $\O(x,T)=\{tx: t\in T\}$ denotes the
{\em orbit} of $x$, which is also denoted by $T x$. We usually denote the closure of $\O(x,T)$ by $\overline{\O}(x,T)$, or $\overline{Tx}$. A subset
$A\subseteq X$ is called {\em invariant} if $t a\subseteq A$ for all
$a\in A$ and $t\in T$. When $Y\subseteq X$ is a closed and invariant subset of the flow $(X, T)$ we say that the flow $(Y, T)$ is a {\em subflow} of $(X, T)$. If $(X, T)$ and $(Y,T)$ are two flows, their {\em product flow} is the
flow $(X \times Y, T)$, where $t(x, y) = (tx, ty)$ for any $t\in T$ and $x,y\in X$.
For $n \geq 2$ we write $(X^n,T)$ for the $n$-fold product flow $(X\times
\cdots \times X,T)$.

\medskip

A flow $(X,T)$ is called {\em minimal} if $X$ contains no proper non-empty
closed invariant subsets. A point $x\in X$ is called a {\em minimal point} or an {\em almost periodic point} if $(\overline{\O}(x,T), T)$ is a mimimal flow.

A flow $(X,T)$ is called {\em transitive} if every invariant open subset of $X$ is dense; and it is {\em point transitive} if there is a point with a dense orbit. It is easy to verify that a flow is
minimal if and only if every orbit is dense.


The flow $(X,T)$ is {\em weakly mixing} if the product flow $(X
\times X,T)$ is transitive.

\medskip

A {\it factor map} $\pi: X\rightarrow Y$ between the flow $(X,T)$
and $(Y, T)$ is a continuous onto map which intertwines the
actions; we say that $(Y, T)$ is a {\it factor} of $(X,T)$ and
that $(X,T)$ is an {\it extension} of $(Y,S)$.
The flows are said to be {\it isomorphic} if $\pi$ is bijective.

\medskip

\medskip

Let $(X,T)$ be a flow. Fix $(x,y)\in X^2$. It is a {\it proximal} pair if $\inf_{t\in T} d(tx, ty)=0$; it is a {\it distal} pair if it is not proximal. Denote by $P(X,T)$ the set of proximal pairs of $(X,T)$. $P(X,T)$ is also called the proximal relation. A flow $(X,T)$ is {\it distal} if $P(X,T)= \D(X)$. A flow $(X,T)$ is {\em equicontinuous} if for any $\ep>0$, there is a $\d>0$ such that whenever $x,y\in X$ with  $d(x,y)<\d$,
then $d(tx,ty)<\ep$ for all $t\in T$.
Any equicontinuous flow is distal.

Let $(X,T)$ be a flow. There is a smallest invariant equivalence relation $S_{eq}$ such that the quotient flow $(X/S_{eq},T)$ is equicontinuous \cite{EG60}. The equivalence relation $S_{eq}$ is called the {\em equicontinuous structure relation} and the factor $(X_{eq}=X/S_{eq}, T)$ is called the {\em maximal equicontinuous factor} of $(X,T)$.


\subsection{Enveloping semigroups}\
\medskip

Given a flow $(X,T),$ its {\em enveloping semigroup} or {\em Ellis semigroup} $E(X,T)$ is defined as the closure of the set $\{t: t\in T\}$ in $X^X$ (with its compact, usually non-metrizable, pointwise
convergence topology). For a enveloping semigroup, $E(X,T)\rightarrow E(X,T):$ $q\mapsto qp$ and $p\mapsto tp$ is continuous for all $p\in E(X,T)$ and $t\in T$. Note that $(X^X,T)$ is a flow and $(E(X,T),T)$ is its subflow.





\begin{de}
A set $E$ is an $\E$-semigroup if it
satisfies the following three conditions:
\begin{enumerate}
  \item $E$ is a semigroup.
  \item $E$ has a compact Hausdorff topology.
  \item The right translation map $R_p: E \longrightarrow E,
q\longmapsto qp$ is continuous for every $p \in E$.
\end{enumerate}
\end{de}

It is easy to see that for a flow $(X,T)$ the enveloping
semigroup $E(X,T)$ is an $\E$-semigroup.

\medskip

For a semigroup the element $u$ with $u^2=u$ is called {\em
idempotent}.
Ellis-Numakura Theorem says that for any $\E$-semigroup $E$ the set $J(E)$ of idempotents of $E$ is not empty \cite{Ellis}.



\medskip

Let $E$ be an $\E$-semigroup. A non-empty subset $I \subseteq E$  is a {\em left ideal}
if $EI \subseteq I$. A {\it minimal left ideal} is a
left ideal that does not contain any proper left ideal of $E$.
Every left ideal is a semigroup and every left ideal
contains some minimal left ideal.

\begin{prop}\cite[Chapter 6]{Au88} \label{prop-proximal}
Let $(X,T)$ be a flow, $x,y\in X$.
\begin{enumerate}
  \item For each $u\in J(E(X,T))$, $(x,ux)\in P(X,T)$.

  \item A pair $(x,y)\in P(X,T)$ if and only if $px=py$ for some $p\in E(X,T)$, if and
only if there is some minimal left ideal $I$ of $E(X,T)$ such that $px=py$ for every $p\in I$.

  \item If $(X,T)$ is a minimal flow, then $(x,y)\in P(X)$ if and only if there is a minimal idempotent $u$ such that $y=ux$.
\end{enumerate}

\end{prop}

\subsection{Universal point transitive flow and universal minimal flow}\
\medskip

For a fixed $T$,  there exists a universal point transitive flow
$(S_T,T)$ such that $T$ can densely and
equivariantly be embedded in $S_T$ \cite[Chapter 8]{Au88}. The multiplication on $T$ can
be extended to a multiplication on $S_T$, then $S_T$ is an $\E$-semigroup. The universal minimal
flow $\mathfrak{M}=(\M,T)$ is isomorphic to any minimal left
ideal in $S_T$ and $\M$ is also an $\E$-semigroup. Hence $J=J(\M)$ of idempotents in $\M$ is nonempty. For $x\in X$, set $J_x=\{u\in J: ux=x\}$.

\begin{prop}\cite[Chapter 6]{Au88}\label{prop-minimal-ideal}
\begin{enumerate}
\item  For $v\in J$ and $p\in \M$, $pv=p$.
\item For each $v\in J$,\ $v\M=
\{vp:p\in \M\}=\{p\in \M: vp=p\}$ is a subgroup of $\M$ with
identity element $v$.
For every $w\in J$ the map $p\mapsto wp$
is a group isomorphism of $vI$ onto $wI$.
\item $ \{v\M:v\in J\}$ is a partition
of $\M$. Thus if $p\in \M$ then there exists a unique
$v\in J$ such that $p\in v\M$.
\end{enumerate}
\end{prop}

Since $v\M$ is a group ($v\in J$), for $p\in \M$, we denote $(vp)^{-1}$ as the inverse of $vp$ in $v\M$.

\medskip

The sets $S_T$ and $\M$ act on $X$ as semigroups and $S_T
x=\overline{T x}$, while for a minimal flow $(X,T)$ we have $\M
x=\overline{T x}=X$ for every $x\in X$.

\begin{prop}\cite[Chapter 6]{Au88}\label{prop-minimal-point}
Let $(X, T)$ be a flow and $x\in X$.
A necessary and sufficient condition for $x$ to be minimal is that $ux=x$ for some $u\in J$.
\end{prop}

Thus for a closed invariant subset $A$ of $X$, $JA=\{ua: u\in J, a\in A\}$ is the set of all minimal points contained in $A$.

\subsection{Hyperspace flow and circle operation}\
\medskip

Let $X$ be a compact metric space. Let $2^X$ be the collection of nonempty closed subsets of $X$
endowed with the Hausdorff topology. One may define a metric on $2^X$ as follows:
\begin{equation*}
\begin{split}
 d_H(A,B)& = \inf \{\ep>0: A\subseteq B_\ep[B], B\subseteq B_\ep[A]\}\\
 &= \max \{\max_{a\in A} d(a,B),\max_{b\in B} d(b,A)\},
\end{split}
\end{equation*}
where $d(x,A)=\inf_{y\in A} d(x,y)$ and $B_\ep [A]=\{x\in X: d(x, A)<\ep\}$.
The metric $d_H$ is called the {\em Hausdorff metric} of $2^X$.

\medskip

Let $(X,T)$ be a flow. We can induce a flow on $2^X$. The
action of $T$ on $2^X$ is given by $tA=\{ta:a\in A\}$ for each $t\in T$ and $A \in 2^X$. Then
$(2^X,T)$ is a flow and it is called the {\em hyperspace flow}.

As $(2^X, T)$ is a flow, $S_T$ acts on $2^X$ too. To avoid ambiguity we
denote the action of $S_T$ on $2^X$ by the {\em circle operation}
as follows. Let $p\in S_T$ and $D\in 2^X$, then define $p\c
D=\lim_{2^X} t_i D$ for any net $\{t_i\}_i$ in $T$ with $t_i\to p$.
Moreover,
\begin{equation*}
    p\c D=\{x\in X: \text{there are $d_i\in D$ with $x=\lim_i t_id_i$}\}
\end{equation*}
for any net $t_i\to p$ in $S_T$. We always have $pD\subseteq p\c
D$.

Note that if $A\in 2^X$ is finite and $p\in S_T$, then $pA=p\c A$.

\subsection{Almost one to one extensions and O-diagram}\
\medskip

Let $(X,T)$ and $(Y, T)$ be flows and let $\pi: X \to Y$ be a factor map.
One says that:
\begin{enumerate}
  \item $\pi$ is an {\it open extension} if it is open as a map;
  \item $\pi$ is an {\it almost one to one extension}   if there exists a dense $G_\d$ set $X_0\subseteq X$ such that $\pi^{-1}(\{\pi(x)\})=\{x\}$ for any $x\in X_0$.
\end{enumerate}

The following is a well known fact about open mappings (see \cite[Appendix A.8]{Vr} for example).

\begin{thm}\label{thm-open}
	Let $\pi:(X,T)\rightarrow(Y,T)$ be a factor map of flow. Then the map
$	\pi^{-1}:Y\rightarrow 2^X, y\mapsto \pi^{-1}(y)$
	is continuous if and only if $\pi$ is open.
\end{thm}

\medskip

Every extension of minimal flows can be lifted to an open
extension by almost one-to-one modifications (\cite{V70, AuG77} or \cite[Chapter VI]{Vr}). To be precise,
\begin{thm}
For every extension $\pi:X\rightarrow Y$ of minimal flows there
exists a commutative diagram of extensions
(called the {\em O-diagram})
\begin{equation*}
\xymatrix
{
X \ar[d]_{\pi}  &  X^* \ar[l]_{\sigma}\ar[d]^{\pi*} \\
Y &  Y^*\ar[l]^{\tau}
}
\end{equation*}

with the following properties:
\begin{enumerate}
\item[(a)]
$\sigma$ and $\tau$ are almost one-to-one;
\item[(b)]
$\pi^*$ is an open extension;
\item[(c)]
$X^*$ is the unique minimal set in $R_{\pi
\tau}=\{(x,y)\in X\times Y^*: \pi(x)=\tau (y)\}$ and $\sigma$ and
$\pi^*$ are the restrictions to $X^*$ of the projections of
$X\times Y^*$ onto $X$ and $Y^*$ respectively.
\end{enumerate}
\end{thm}

\medskip

We sketch the construction of these factors. Let $x \in X$, $u \in
J_{x}$ and $y=\pi (x)$. Let $y^*=u\c \pi^{-1}(y).$ One has that $y^*$ is a minimal point of $(2^X, T)$ and define $Y^*=\{p\c y^*: p \in \M\}$ as the orbit closure of $y^*$ in $2^X$
for the action of $T$. Finally $X^*=\{(px, p\c y^*)\in X\times Y^*:
p\in \M \}$, $\tau (p\c y^*)= p y$ and $\sigma ((px,p\c y^*))
=px$. It can be proved that $X^*=\{(\tilde x,\tilde y)\in X \times
Y^* : \tilde x \in \tilde y\}$.

\medskip

There is another equivalent way to get O-diagram. Let $\pi^{-1}: Y\rightarrow 2^X, y\mapsto \pi^{-1}(y)$. Then  $\pi^{-1}$ is a u.s.c. map, and the set $Y_c$ of
continuous points of $\pi^{-1}$ is a dense $G_\d$ subset of $Y$.
Let $$\widetilde{Y}=\overline{\{\pi^{-1}(y): y\in Y\}}\ \text{and}\ Y^*=\overline{\{\pi^{-1}(y): y\in Y_c\}},$$
where the closure is taken in $2^X$. It is obvious that $Y^*\subseteq \widetilde{Y}\subseteq 2^X$. Note that
for each $A\in \widetilde{Y}$, there is some $y\in Y$ such that $A\subseteq \pi^{-1}(y)$, and hence $A\mapsto y$
define a map $\tau: \widetilde{Y}\rightarrow Y$. It is easy to verify that $\tau: (\widetilde{Y},T)\rightarrow (Y,T)$
is a factor map. One can show that if $(Y,T)$ is minimal then $(Y^*,T)$ is a minimal flow and it is the unique minimal subflow in $(\widetilde{Y},T)$, and $\tau: Y^*\rightarrow Y$ is an almost one to one extension. When $Y^*$ is defined, $X^*, \sigma$ and $\pi^*$ are defined as above.


\subsection{Proximal extensions and RIC-diagram}\
\medskip

Let $(X,T)$ and $(Y, T)$ be flows and let $\pi: X \to Y$ be a factor map.
One says that:
\begin{enumerate}
   \item $\pi$ is  a {\it distal extension}  if $\pi(x_1)=\pi(x_2)$ and $x_1\neq x_2$ implies $(x_1,x_2) \not\in P(X,T)$;
 \item $\pi$ is an {\it equicontinuous or isometric extension}  if for any $\ep >0$ there exists $\d>0$ such that $\pi(x_1)=\pi(x_2)$ and $d(x_1,x_2)<\d$ imply $d(tx_1,tx_2)<\ep$ for any $t\in T$.

  \item $\pi$ is a {\em weakly mixing extension}  if $(R_\pi, T)$ as a subflow of the product
flow $(X\times X, T)$ is transitive, where $R_\pi=\{(x_1,x_2)\in X^2: \pi(x_1)=\pi(x_2)\}$.

\end{enumerate}

Let $\pi: (X,T )\rightarrow (Y,T)$ be a factor map of minimal flows, and $x_0\in X$, $y_0=\pi(x_0)$. We say that $\pi$
is a {\em RIC} (relatively incontractible) extension if for every $y
= py_0\in Y$, $p\in\M$,
$$\pi^{-1}(y)=p\c u\pi^{-1}(y_0).$$
One can show that $\pi : X \to
Y$ is RIC if and only if it is open and for every $n \ge 1$
the minimal points are dense in the relation
$$
R^n_\pi = \{(x_1,\dots,x_n) \in X^n : \pi(x_i)=\pi(x_j),\ \forall \ 1\le i
\le j \le n\}.
$$

Note that every
distal extension is RIC, and every distal extension
is open.

\medskip

Every factor map between minimal flows can be
lifted to a RIC extension by proximal extensions (see \cite{EGS} or \cite[Chapter VI]{Vr}).

\begin{thm}\label{RIC}
Given a factor map $\pi:X\rightarrow Y$ of minimal flows, there exists a commutative diagram of factor maps (called {\em RIC-diagram})
\begin{equation*}
\xymatrix
{
X \ar[d]_{\pi}  &  X' \ar[l]_{\sigma'}\ar[d]^{\pi'} \\
Y &  Y'\ar[l]^{\tau'}
}
\end{equation*}
	such that:
	\begin{enumerate}
		\item[(a)] $\tau '$ and $\sigma'$ are proximal extensions;
		\item[(b)] $\pi '$ is a RIC extension;
		\item[(c)] $X '$ is the unique minimal set in $R_{\pi \tau'
		}=\{(x,y)\in X\times Y ': \pi(x)=\tau'(y)\}$, and $\sigma'$ and
		$\pi '$ are the restrictions to $X '$ of the projections of $X\times
		Y '$ onto $X$ and $Y '$ respectively.
	\end{enumerate}
\end{thm}

We sketch the construction of these factors. Let $x\in
X$, $u\in J_{x}$ and $y=\pi (x)$. Let $y'=u \c u \pi^{-1}(y)$,
then $y'$ is a minimal point in $2^X$.
Define $Y'=\{p\c y': p\in \M\}$ to be the orbit closure of $y'$
and $X'=\{(px, p\c y')\in X\times Y': p\in \M\}$, and factor maps
given by $\tau' (p\c y')= p y$ and $\sigma'((px,p\c y')) =px$. It
can be proved that $X'=\{(\tilde x,\tilde y)\in X \times Y' :
\tilde x \in \tilde y\}$.



\section{almost proximal extensions}\label{section-ap}

In this section, we introduce almost proximal extensions. We will give the structure of almost proximal extensions, and study its relationship with almost finite to one extensions.

\subsection{Almost proximal extensions}\
\medskip

An {\em almost periodic set} for $(X,T)$ is subset $A$ of $X$ such that if $z\in X^{|A|}$ with ${\rm range} (z)=A$, then $z$ is a minimal point of the flow $(X^{|A|},T)$. ($|A|$ denotes the cardinality of $A$.) For example, a finite set $A=\{x_1,x_2,\ldots, x_n\}$ is an almost periodic set if and only if $(x_1,x_2,\ldots, x_n)$ is a minimal point of $(X^n,T)$. Using the basis for the Tychonoff topology, we see that a set $A$ is an almost periodic
set if and only if every finite subset of $A$ is an almost periodic set. The notion of almost periodic set was introduced by Auslander, refer to \cite[Chpter 5]{Au88} for more information.

\begin{de}
Let $\pi: (X,T)\rightarrow (Y,T)$ be an extension of minimal flows. $\pi$ is called an {\em almost proximal extension} if for each $y\in Y$, there is some $N\in \N$ such that the cardinality of any almost periodic subset in the fiber $\pi^{-1}(y)$ is not greater than $N$.
\end{de}

Let $A$ be an almost periodic set of $(X,T)$. Then for each $z\in X^{|A|}$ with ${\rm range} (z)=A$, $z$ is a minimal point of $(X^{|A|},T)$. By Proposition \ref{prop-minimal-point}, there is some $u\in J$ such that $uz=z$. It follows that $uA=\{ua: a\in A\}=A$. Thus, for a subset $Z$ of $X$, a subset $A\subseteq Z$ is an almost periodic subset of $Z$ if and only if $A\subseteq uZ$ for some $u\in J$.

\begin{lem}\label{lem-number}
Let $\pi: (X,T)\rightarrow (Y,T)$ be an extension of minimal flows. Let $y, y' \in Y$, $u \in J_y$ and $v \in J_{y'}$.
Then $|u\pi^{-1}(y)|=|v\pi^{-1}(y')|$.
\end{lem}

\begin{proof}
We only show the case when $u\pi^{-1}(y), v\pi^{-1}(y')$ are finite. The same proof works for general case.
Let $u\pi^{-1}(y)=\{x_1,x_2,\ldots, x_n\}$ and $v\pi^{-1}(y')=\{x_1',x_2',\ldots, x_m'\}$
for some $n,m\in \N$. Then $(x_1,x_2,\ldots, x_n)\in X^n$ is a minimal point of $(X^n, T)$ as $u(x_1,x_2,\ldots, x_n)=(x_1,x_2,\ldots, x_n)$. Since $(X,T)$ is minimal, there is $p\in \M$ such that $x_1'=px_1$. Note that $x_1'=vx_1'=vpx_1$ and $x_1,x_2,\ldots, x_n\in \pi^{-1}(y)$. It follows that
$$y'=\pi(x_1')=\pi(vpx_1)=\ldots =\pi(vpx_n)=vp\pi(x_1)=vpy.$$
Since $x_1,x_2,\ldots,x_n$ are distinct and $(x_1,x_2,\ldots, x_n)$ is minimal, $vpx_1, vpx_2, \ldots, vpx_n$ are also distinct. As
$$vpx_1, vpx_2, \ldots, vpx_n \in v\pi^{-1}(y')=\{x_1',x_2',\ldots, x_m'\},$$
we have $n\le m$. Similarly, we have $m\le n$. Thus
$|u\pi^{-1}(y)|=|v\pi^{-1}(y')|$.
\end{proof}

By Lemma \ref{lem-number} and the fact that $A$ is an almost periodic subset of subset $Z$ if and only if $A\subseteq uZ$ for some $u\in J$, we have the following proposition readily.

\begin{prop}\label{a.p equivalenve}
Let $\pi: (X,T)\rightarrow (Y,T)$ be an extension of minimal flows. Then the following are equivalent:
\begin{enumerate}
  \item $\pi$ is {almost proximal}.
  \item For some $y\in Y$, there is some $N\in \N$ such that the cardinality of any almost periodic subset in the fiber $\pi^{-1}(y)$ is not greater than $N$.
  \item For some $y\in Y$ and $u\in J_y$, $|u\pi^{-1}(y)|=N<\infty$.
  \item For each $y\in Y$ and $u\in J_y$, $|u\pi^{-1}(y)|=N<\infty$.
\end{enumerate}

When $N=1$, $\pi$ is proximal.
\end{prop}

For $x\in X$, $P[x]=\{x'\in X: (x,x')\in P(X)\}$ is called the {\em proximal cell} of $x$. It is clear that $\pi$ is proximal if and only if for any $y\in Y$ and any $x\in \pi^{-1}(y)$, $\pi^{-1}(y)\subseteq P[x]$.

\begin{cor}
Let $\pi: (X,T)\rightarrow (Y,T)$ be an extension of minimal flows. If $\pi$ is {almost proximal}, then for each $y\in Y$, there is some finite subset $F$ of $\pi^{-1}(y)$ such that each point of $\pi^{-1}(y)$ is proximal to some point of $F$, i.e., $\pi^{-1}(y)\subseteq \bigcup_{x\in F}P[x]$.
\end{cor}

\begin{proof}
Let $y\in Y$ and $u\in J_y$. Then $F=u\pi^{-1}(y)$ is finite by Proposition \ref{a.p equivalenve}. For any $x\in \pi^{-1}(y)$, $(x,ux)\in P(X)$ and $ux\in F$. The proof is complete.
\end{proof}


Let $\pi:(X,T)\rightarrow (Y,T)$ be an extension  of flows $(X,T)$ and $(Y,T)$ and $n\in \N$.
Let $$2^X_\pi=\{A\in 2^X: A\subseteq \pi^{-1}(y) \ \text{for some}\ y\in Y \},$$
and
$$2^X_{\pi, n}=\{A\in 2^X_\pi: |A|\le n\}, \ \text{and}\
 2^X_{\pi,*}=\{A\in 2^X_\pi: |A|< \infty\}=\bigcup_{n=1}^\infty 2^X_{\pi, n}.$$
It is clear that $(2^X_\pi, T)$, $(2^X_{\pi,n}, T)$ and $(2^X_{\pi,*}, T)$ are subflows of $(2^X,T)$.
Let
$$\widetilde{\pi}: (2^X_\pi, T) \rightarrow (Y,T), \ A\mapsto \pi(A).$$
Then $\widetilde{\pi}$ is an extension. Note that
$2^X_{\pi,1}=\{\{x\}\in 2^X: x\in X\}$ is isomorphic to $X$ and $\widetilde{\pi}|_{2^X_{\pi,1}}$ is the same to $\pi$.
Thus $\pi: (X,T) \rightarrow (Y,T)$ is proximal if and only if $\widetilde{\pi}|_{2^X_{\pi,1}}: (2^X_{\pi,1}, T)\rightarrow (Y,T)$ is proximal.

\begin{prop}
Let $\pi:(X,T)\rightarrow (Y,T)$ be an extension of minimal flows. Then $\pi$ is almost proximal if and only if there is some $n\in \N$ such that there are no minimal points in $2^X_{\pi,*}\setminus 2^X_{\pi, n}$.
\end{prop}

\begin{proof}
If $\pi$ is almost proximal, then there is some $n\in \N$ such that $|u\pi^{-1}(y)|=n$ for all $y\in Y, u\in J_y$. Let $A\in 2^X_{\pi, *} \setminus 2^X_{\pi, n}$ and $y=\pi(A)$. Then $|A|\ge n+1$. For each $u\in J_y$, $|uA|\le n$ since $\pi$ is almost proximal. Thus $u\circ A=uA\neq A$, in particular, by Proposition \ref{prop-minimal-point}, $A$ is not minimal.

\medskip
Conversely, assume that there is some $n\in \N$ such that there are no minimal points in $2^X_{\pi,*}\setminus 2^X_{\pi, n}$. If $\pi$ is not almost proximal. Then there is some $y\in Y$ and $u\in J_y$ such that $|u\pi^{-1}(y)|=\infty$.
We choose $A\subseteq u\pi^{-1}(y)$ with $|A|= n+1$. Then $$u\circ A=uA=A,$$
which means $A$ is a minimal point of $2^X_{\pi,*}\setminus 2^X_{\pi, n}$, a contradiction!
The proof is complete.
\end{proof}

\begin{rem}
It is well known that each flow has a minimal subflow. Since $2^X_{\pi,*}\setminus 2^X_{\pi, n}$ may be not compact, it is an invariant subset of $2^X$ but maybe not subflow of $2^X$. So it is possible that $2^X_{\pi,*}\setminus 2^X_{\pi, n}$ contains no minimal points.
\end{rem}

\subsection{Almost finite to one extensions}\label{subsec-finite-one}\
\medskip

\begin{de}
Let $\pi:(X,T)\rightarrow (Y,T)$ be an extension of minimal flows. $\pi$ is called an {\em almost finite to one} extension if  some fiber is finite.
\end{de}

\begin{prop}\label{almost N-1}\cite{Shoenfeld, HLSY}
Let $\pi:(X,T)\rightarrow (Y,T)$ be an extension of minimal flows.
The following statements are equivalent:
\begin{enumerate}
  \item $\pi$ is almost finite to one, i.e., some fiber is finite;
  \item There exists $N\in \N$ such that $Y_0=\{y\in Y: |\pi^{-1}(y)|=N\}$ is a residual subset of $Y$;
  \item There exist $N\in \N$ and $y_0\in Y$ such that $|\pi^{-1}(y_0)|=N$ and $\pi^{-1}(y_0)$ is an almost periodic set.
  \item The cardinality of each minimal point of $(2^X,T)$ in $2^X_{\pi}$ is not greater than some fixed integer $N$.
\end{enumerate}
\end{prop}


\begin{rem}
\begin{enumerate}
  \item By Proposition \ref{almost N-1}, almost finite to one extensions of minimal flows are almost proximal. In Section \ref{section-example}, we will give examples which are almost proximal but not almost finite to one extensions.
  \item By definition it is obvious that a finite to one extension (i.e., every fiber is finite) is almost finite to one. But in general, an almost finite to one extension may not be finite to one. For example, for Rees' example \cite{R}, $\pi: (X,T)\rightarrow (X_{eq},T)$ is an almost one to one extension ($X_{eq}$ is the maximal equcontinuous factor of $X$), and for any $y\in X_{eq}$, either $|\pi^{-1}(y)|=1$ or $|\pi^{-1}(y)|=\infty$.
  \item For more discussion about finite to one and almost finite to one extensions, refer to \cite{HLSY}.
\end{enumerate}

\end{rem}


By Proposition \ref{almost N-1}, it is easy to check the following corollary.

\begin{cor}\label{cor-2.17}
Let $\pi:(X,T)\rightarrow (Y,T)$ be an extension of minimal flows.
If $\pi$ is proximal but not almost one to one, then every fiber is infinite, i.e., $\pi$ is not almost finite to one.
\end{cor}




\subsection{The structure of almost proximal extensions}\
\medskip

The following result is well known and it is easy to be verified (one may find a proof in \cite{HLSY}).
\begin{lem} \label{finite-to-one}
	Let $\pi: X \rightarrow Y$ be a finite to one extension (i.e., $\pi^{-1}(y)$ is finite for all $y\in Y$)
	of the minimal flows $(X,T)$ and $(Y,T)$. Then the following conditions are equivalent:
	\begin{enumerate}
		\item $\pi$ is open;
		\item $\pi$ is distal;
		\item $\pi$ is equicontinuous;
	\end{enumerate}
	In this case there exists $N\in \N$ such that $\pi$ is an $N$ to one map.
\end{lem}

\begin{thm}\label{GP-Structure}
Let $\pi: (X,T)\rightarrow (Y,T)$ be an almost proximal extension of minimal flows. Then it has the
following structure:
\begin{equation*}
\xymatrix
{
X \ar[d]_{\pi}  &  X' \ar[l]_{\sigma'}\ar[d]^{\pi'} \\
Y &  Y'\ar[l]^{\tau'}
}
\end{equation*}
where $\sigma'$ and $\tau'$ are proximal, $\pi'$ is a
finite to one equicontinuous extension.

Moreover, $\pi$ is almost finite to one if and only if $\tau', \sigma'$ are almost one to one.
\end{thm}

\begin{proof}
In light of the construction of RIC-diagram, $\sigma'$ and $\tau'$ are proximal and $\pi '$ is a RIC extension. Let $y\in Y$ and $u\in J_y$.	Since $\pi$ is almost proximal, then $|u\pi^{-1}(y)|<\infty.$ Thus $y'=u\circ u\pi^{-1}(y)$ is finite, and for every $p\in \M,$ $|p\c y'  |<\infty.$ It follows that $\pi'$ is finite to one and open. By Lemma \ref{finite-to-one}, $\pi '$ is equicontinous.
	
\medskip

	Clearly if $\tau', \sigma'$ are almost one to one then  $\pi$ is almost finite to one. Conversely, when $\pi$ is almost finite to one,
    by Proposition \ref{almost N-1}, there exists $y_0$ such that $|\pi ^{-1}(y_0)|< \infty$ and $\pi ^{-1}(y_0)$ is a minimal point of $(2^X,T),$i.e., $u\circ \pi ^{-1}(y_0)=\pi ^{-1}(y_0).$ Hence the construction of RIC-diagram coincides with O-diagram. Therefore, $\tau', \sigma'$ are almost one to one.
\end{proof}

If in addition the extension is regular, an almost proximal extension has a succinct structure. Let
${\rm Aut} (X,T)$ be the group of automorphisms of the flow $(X,T)$, that is, the group of all self-homeomorphisms $\psi$ of $X$ such that $\psi\circ t=t\circ \psi, \forall t\in T$. For an extension $\pi: (X,T)\rightarrow (Y,T)$, let $${\rm Aut}_\pi(X,T)=\{\chi\in {\rm Aut}(X,T): \pi\circ \chi=\pi\},$$ i.e., elements of ${\rm Aut}(X,T)$ mapping every fiber of $\pi$ into itself.

\begin{de}
Let $\pi: (X, T)\rightarrow (Y,T)$ be an extension of minimal flows. One says $\pi$ is {\em regular} if for any point $(x_1,x_2)\in R_\pi$ there exists  $\chi \in {\rm Aut}_\pi(X,T)$ such
that $(\chi(x_1),x_2)\in P(X,T)$. It is equivalent to: for any
minimal point $(x_1,x_2)$ in $R_\pi$ there exists $\chi\in
{\rm Aut}_\pi (X,T)$ such that $\chi(x_1)=x_2$.
\end{de}

The notion of regularity was introduced by Auslander \cite{Au66}. Examples of regular extensions are
proximal extensions, group extensions. For more information about regularity, refer to \cite{Au66, Au88, G92}.


\begin{thm}\label{thm-reg}
Let $\pi: (X,T)\rightarrow (Y,T)$ be an extension of minimal flows.
If $\pi$ is regular and almost proximal, then there exists a flow $(Y^\#,T)$ with the following commutative diagram
$$
\xymatrix@R=0.5cm{
  X \ar[dd]_{\pi} \ar[dr]^{\pi^\#}             \\
                & Y^\# \ar[dl]_{\tau^\#}         \\
  Y                 }
$$
where $\tau^\#$ is proximal and $\pi^\#$ is a
finite to one equicontinuous extension. And $\pi$ is almost finite to one if and only if $\tau^\#$ is almost one to one.
\end{thm}

\begin{proof}
Let $y_0\in Y$ and $u\in J_{y_0}$, i.e., $uy_0=y_0$. Let ${\bf y}_0=u\pi^{-1}(y_0)=\{x_1,x_2,\ldots, x_n\}$ where $n\in \N$.
Let $$Y^\#=\{p{\bf y}_0: p\in \M\}.$$
It is clear that $X=\bigcup_{{\bf y}\in Y^\#} {\bf y}$.
We show that $Y^\#$ is a partition of $X$, that is, for all $p,q\in \M$, either $p{\bf y}_0=q{\bf y}_0$, or $p{\bf y}_0 \cap q{\bf y}_0=\emptyset$.

Assume that there are $p,q\in \M$ such that $p{\bf y}_0=\{px_1,\ldots, px_n\}\not =\{qx_1, \ldots, qx_n\}=q{\bf y}_0$ and $p{\bf y}_0 \cap q{\bf y}_0\neq \emptyset$. Without loss of generality, we assume that $px_1=qx_1$ and $qx_2\not\in p{\bf y}_0$.
Since $(qx_1,qx_2)\in R_\pi$ is a minimal point and $\pi$ is regular, there is some $\chi\in {\rm Aut}_\pi(X,T)$ such that
$qx_2=\chi(qx_1)$. Let $v\in J(\M)$ such that $vp{\bf y}_0=p{\bf y}_0$. Then
$$qx_2=\chi(qx_1)=\chi(px_1)=\chi(vpx_1)=v\chi(px_1)=v\chi(qx_1)=vqx_2.$$
Thus $qx_2\in v\pi^{-1}(py_0)$. It follows that
$\{px_1,px_2,\ldots, px_n, qx_2\}$ is an almost periodic subset of $\pi^{-1}(py_0)$, whose cardinality is $n+1$.
This contradicts with Lemma \ref{lem-number}. Thus $Y^\#$ is a partition of $X$. It induces a map $$\pi^\#: X\rightarrow Y^\#, \ x\mapsto p{\bf y}_0, \ (x\in p{\bf y}_0), p\in \M.$$
It is easy to check that $\pi$ is open, and by Lemma \ref{finite-to-one} it is a $n$ to $1$ equicontinuous extension.

Let $$\tau^\# : Y^\#\rightarrow Y, p{\bf y_0}\mapsto py_0, \forall p\in \M.$$
We show that it is proximal. Let $\tau^\#(p{\bf y}_0)=\tau^\#(q{\bf y}_0)$ for some $p,q\in \M$. Then
$py_0=qy_0$. There are minimal idempotents $u_1,u_2\in J(\M)$ such that
\[p=u_1 p,q=u_2q.\]
Since $|{\bf y}_0|<\infty,$ we have
\[p{\bf y}_0=u_1p{\bf y}_0=u_1p  u\pi^{-1}(y_0)=u_1\pi^{-1}(py_0).\]
Similarly, one has $q{\bf y}_0=u_2\pi^{-1}(qy_0)$.
Then
$$u_1q{\bf y}_0=u_1u_2\pi^{-1}(qy_0)=u_1\pi^{-1}(py_0)=p{\bf y}_0.$$
So $p{\bf y}_0$ and $q{\bf y}_0$ are proximal. That is, $\tau^\#$ is a proximal extension.
\end{proof}


\subsection{Relations between almost proximal extensions and almost one to one extensions}\

It is clear that any almost finite to one extension is an almost proximal extension. In this subsection we show that if in addition the extension is point distal, then the converse holds.

Let $\pi: (X,T)\rightarrow (Y,T)$ be an extension of minimal flows. A point $x\in X$ is called a {\em $\pi$-distal point} whenever $P_\pi[x]\triangleq\{x'\in \pi^{-1}(\pi(x)): (x,x')\in P(X)\}=\{x\}$, that is, $(x,x')$ is a distal pair for every $x'\in \pi^{-1}(\pi(x))\setminus \{x\}$. The extension $\pi$ is said to be {\em point distal} when there exists a $\pi$-distal point in $X$. Note that $\pi$ is distal if and only if every point is $\pi$-distal.
It is easy to see that a point $x$ is $\pi$-distal if and only if $ux=x$ for all $u\in J_{\pi(x)}$ \cite[Chapter VI-(4.3)]{Vr}.

\medskip

\begin{thm}\cite[Theorem 2.3.5.]{V77}\label{thm-Veech}
Let $\pi: (X,T)\rightarrow (Y,T)$ be an extension of minimal flows. Then RIC-diagram and O-diagram coincide if and only if there is some $y\in Y$ such that $\bigcap_{u\in J_y} u\circ u\pi^{-1}(y)\neq \emptyset$.
\end{thm}

\begin{thm}
Let $\pi: (X,T)\rightarrow (Y,T)$ be an extension of minimal flows.
Then the following are equivalent:
\begin{enumerate}
  \item $\pi$ is almost proximal and point distal;
  \item $\pi$ is almost finite to one.
\end{enumerate}
\end{thm}

\begin{proof}
If $\pi$ is almost finite to one, it is almost proximal. By Theorem \ref{almost N-1}, there exist $N\in \N$ and $y_0\in Y$ such that $|\pi^{-1}(y_0)|=N$ and $\pi^{-1}(y_0)$ is a minimal point of $(2^X, T)$. Each point in $\pi^{-1}(y_0)$ is $\pi$-distal.

Conversely, assume that $\pi$ is almost proximal and point distal. Let $x_0\in X$ be a $\pi$-distal point and $y_0=\pi(x_0)$. Since $x_0$ is $\pi$-distal, $ux_0=x_0$ for all $u\in J_{y_0}$. Thus $x_0\in \cap_{u\in J_{y_0}}u\circ u \pi^{-1}(y_0)$. By Theorem \ref{thm-Veech}, the RIC-diagram and O-diagram of $\pi$ coincide. That is,  in the RIC-diagram,
\begin{equation*}
\xymatrix
{
X \ar[d]_{\pi}  &  X' \ar[l]_{\sigma'}\ar[d]^{\pi'} \\
Y &  Y'\ar[l]^{\tau'}
}
\end{equation*}
$\sigma'$ and $\tau'$ are almost one to one, $\pi'$ is a finite to one extension.
Thus $\pi$ is almost finite to one.
\end{proof}

\begin{cor}
Let $\pi: (X,T)\rightarrow (Y,T)$ be an extension of minimal flows.
Then $\pi$ is proximal and point distal if and only if $\pi$ is almost one to one.
\end{cor}

\section{Proximal extensions and weakly mixing extensions}\label{section-wm-ap}

In this section we will study the chaotic properties of proximal but not almost one to one extensions.

\subsection{Weakly mixing extensions}\label{subsection-wm}\
\medskip

Let $(X,T)$ be a flow. A point $x\in X$ is {\em recurrent} if for each neighbourhood $U$ of $x$, the set of return times of $x$ to $U$, $N(x,U)=\{t\in T: tx\in U\}$ is infinite. A point $x$ is recurrent if and only if there exists $p\in S_T\setminus T$ such that $px=x$, if and only if there exists an idempotent $u\in S_T\setminus T$ such that $ux=x$ \cite{EEN}.

\begin{de}
Let $(X,T)$ be a flow.
\begin{enumerate}
  \item A pair $(x,y)\in X^2$ is called a {\em strong Li-Yorke pair} if it is proximal and it is also a recurrent point of $(X^2, T)$.
  \item A subset $S$ of $X$ is called {\em strongly scrambled} if every pair of distinct points in $S$ is a strong Li-Yorke pair.
  \item The flow $(X,T)$ is said to be {\em strongly Li-Yorke chaotic} if it contains an uncountable strongly scrambled set.
\end{enumerate}
\end{de}

The following result says that each non-trivial weakly mixing extension contains plenty of strongly scrambled subsets.

\begin{thm} \label{thm-wm}\cite{AGHSY}
Let $(X,T), (Y,T)$ be a flow and $\pi: (X,T)\rightarrow (Y, T)$ an open
nontrivial weakly mixing extension. Then there is a residual
subset $Y_0\subseteq Y$ such that for every point $y\in Y_0$ the
set $\pi^{-1}(y)$ contains a dense uncountable strongly scrambled subset $K$ such that
$$K\times K\setminus \D_X\subseteq {\rm Trans}(R_\pi),$$
and in particular, $\pi^{-1}(y)$ is perfect, and it has the cardinality of the continuum.
\end{thm}

\begin{de}
Let $\pi: (X,T)\rightarrow (Y,T)$ be an extension of flows.
\begin{enumerate}
  \item $\pi$ is called {\em $n$-weakly mixing} ($n\ge 2$) if $(R^n_\pi, T)$ is transitive.
  \item If $\pi$ is $n$-weakly mixing for every $n\ge 2$, then $\pi$ is called {\em totally weakly mixing}.
\end{enumerate}
\end{de}

Note that $2$-weakly mixing extension is just weakly mixing. Totally weakly mixing extensions will present stronger chaotic properties than weakly mixing extensions, see \cite{AGHSY, SY18} for more details. For RIC extensions, weak mixing is equivalent to total weak mixing.

\begin{thm}\cite{AMWW, Gl05}\label{thm-total-wm}
Let $\pi: (X,T)\rightarrow (Y,T)$ be a RIC extension of minimal flows. Then the following conditions are equivalent:
\begin{enumerate}
  \item $\pi$ is weakly mixing;
  \item $\pi$ is $n$-weakly mixing for some $n\ge 2$;
  \item $\pi$ is totally weakly mixing.
\end{enumerate}
\end{thm}

In general, the fact that $\pi: (X,T)\rightarrow (Y,T)$ is weakly mixing can not imply that $\pi$ is totally weakly mixing. In \cite{Gl05}, for $(Y,T)$ being trivial, the author gave a weakly mixing flow $(X,T)$ such that $(X^3,T)$ is not transitive. By Theorem \ref{thm-total-wm}, such $T$ is not abelian. In Section \ref{section-example}, we will show that there exists $\pi: (X,\Z)\rightarrow (Y,\Z)$ which is weakly mixing but not 3-weakly mixing. In fact we have

\begin{exam}
There exists a discrete flow $(X,\Z)$ such that $\pi: (X,\Z)\rightarrow (X_{eq}, \Z)$ is a proximal extension, and it is weakly mixing but not $3$-weakly mixing. (See Theorem \ref{thm-exam-wm}).
\end{exam}

\subsection{Proximal but not almost one-to-one extensions}\
\medskip

Let $P^{(n)}$ be the subset of $X^n$ of all points whose orbit closure intersects the diagonal of $X^n$. Thus $P(X)=P^{(2)}$. Glasner showed that for a minimal flow  $(X,T)$, if $P^{(n)}$ is dense in $X^n$ for every $n\ge 2$, then $X$ is weakly disjoint from all minimal flows, i.e., $(X\times Y,T)$ is transitive for all minimal flow $(Y,T)$ \cite[Proposition 2.1.]{G76}. In particular, any minimal proximal flow is weakly mixing. This was extended by van der Woude  \cite{Wo82} as follows (see
also \cite{Gl05}).

\begin{thm}\label{pwm}
A non-trivial open proximal extension of minimal flows is
a weakly mixing extension.
\end{thm}

Thus by Theorem \ref{thm-wm}, we have

\begin{thm}\cite{AGHSY}\label{thm-pr-old}
Let $\pi:X\rightarrow Y$ be a proximal but not almost one-to-one
extension of minimal flows.
Then there is a residual subset $Y_0\subseteq Y$ such
that for each $y\in Y_0$,  $\pi^{-1}(y)$ contains an uncountable strongly scrambled subset
$K$.
\end{thm}

In fact, we have the following result which is slightly stronger than Theorem \ref{thm-pr-old}.

\begin{thm}\label{thm-prox}
Let $\pi:X\rightarrow Y$ be a proximal but not almost one-to-one
extension between minimal flows.
Then there is a residual subset $Y_0\subseteq Y$ such
that for each $y\in Y_0$,  $\pi^{-1}(y)$ contains an uncountable strongly scrambled subset
$K$ satisfying that for any $x_1 \neq x_2\in K$
$$ \pi^{-1}(y)\times \pi^{-1}(y)\subseteq \overline{\O}((x_1,x_2),T).$$
\end{thm}

\begin{proof}
Let $\pi:X\rightarrow Y$ be a proximal but not almost one-to-one
extension between minimal flows.
We consider its O-diagram:
\begin{equation*}
\xymatrix
{
X \ar[d]_{\pi}  &  X^* \ar[l]_{\sigma}\ar[d]^{\pi*} \\
Y &  Y^*\ar[l]^{\tau}
}
\end{equation*}
Recall that
$$Y^*=\overline{\{\pi^{-1}(y): y\in Y_c\}},\quad X^*=\{(x, {\bf y})\in X\times Y^*: x\in {\bf y}\},$$
where $Y_c$ is the set of continuous points of $\pi^{-1}: Y\rightarrow 2^X$.
Since $\pi$ is proximal but not almost one-to-one, $\pi^*$ is open proximal but not almost one-to-one.
By Theorem \ref{pwm}, $\pi^*$ is weakly mixing, and hence by Theorem \ref{thm-wm}, there is a residual subset $Y_0^*$ of $Y^*$ such that for each ${\bf y}\in Y_0^*$,  $(\pi^*)^{-1}({\bf y})$ contains a dense uncountable strongly scrambled subset $K^*$ such that
$$K^*\times K^*\setminus \D_{X^*}\subseteq {\rm Trans}(R_{\pi^*}).$$

Since $\tau$ is almost one-to-one, $\tau(Y_0^*)$ is a residual subset of $Y$. Let
$$Y_0=Y_c\cap \tau(Y_0^*).$$
Then $Y_0$ is a residual subset of $Y$. We need verify that  for each $y\in Y_0$,  $\pi^{-1}(y)$ contains an uncountable strongly scrambled subset $K$ satisfying that for any $x_1 \neq x_2\in K$
$$ \pi^{-1}(y)\times \pi^{-1}(y)\subseteq \overline{\O}((x_1,x_2),T).$$

Let $y\in Y_0$. Since $y\in Y_c\cap \tau(Y_0^*) $, $\tau^{-1}(y)=\pi^{-1}(y)$. Thus ${\bf y}=\pi^{-1}(y)\in Y_0^*$.
So $(\pi^*)^{-1}({\bf y})$ contains a dense uncountable strongly scrambled subset $K^*$ such that
$K^*\times K^*\setminus \D_{X^*}\subseteq {\rm Trans}(R_{\pi^*}).$
Let $$K^*=\{(x_\a, {\bf y}): \a\in \Lambda, x_\a\in {\bf y}\},$$
where $\Lambda$ is an uncountable index set. Now let $K=\sigma(K^*)=\{x_\a\}_{\a\in \Lambda}$.
First note that $$\pi(K)=\pi\sigma(K^*)=\tau\pi^*(K^*)=\tau({\bf y})=y,$$
and hence $K\subseteq \pi^{-1}(y)$.

Let $x_1,x_2\in K$ with $x_1\neq x_2$. We show that $\pi^{-1}(y)\times \pi^{-1}(y)\subseteq \overline{\O}((x_1,x_2),T)$. Let $y_1,y_2\in \pi^{-1}(y)={\bf y}$. Then
$$(y_1, {\bf y}), (y_2, {\bf y})\in (\pi^*)^{-1}({\bf y}).$$
Since $\left((x_1,{\bf y}), (x_2, {\bf y})\right)\in K^*\times K^*\setminus\D_{X^*} \subseteq {\rm Trans}(R_{\pi^*})$,
$$\left((y_1,{\bf y}), (y_2, {\bf y})\right)\in \overline{\O}\left(\big((x_1,{\bf y}), (x_2, {\bf y})\big), T\right).$$
It follows that $(y_1,y_2)\in \overline{\O}((x_1,x_2),T)$.
Thus $ \pi^{-1}(y)\times \pi^{-1}(y)\subseteq \overline{\O}((x_1,x_2),T).$ The proof is complete.
\end{proof}

As mentioned above, any minimal proximal flow is weakly mixing. In fact, for minimal proximal flows, one can say more.

\begin{thm}
Let $(X,T)$ be a proximal minimal flow. Then for every $x\in X$, the set
$$\{y\in X: \overline{\O}((x,y),T)=X\times X\}$$
is residual in $X$.
\end{thm}

\begin{proof}
First we show that for any non-empty open subsets $V , U',V'$ of $X$, there is some $t\in T$ such that
\begin{equation}\label{b1}
 (\{tx\}\times tV)\cap U'\times V'\neq \emptyset.
\end{equation}
Since $(X,T)$ is minimal, there is a finite subset $\{t_1,t_2,\ldots, t_n\}$ of $T$ such that $X=\bigcup_{i=1}^n t_i V$. As $(X,T)$ is proximal, $(t_1x, t_2x, \ldots, t_n x)\in P^{(n)}$. And there is some $t_0\in T$ such that
$$t_0t_1x, t_0t_2 x, \ldots, t_0 t_n x\in U'.$$
Note that $t_0^{-1}V'\cap \bigcup_{i=1}^n t_iV= t_0^{-1}V'\cap X\neq \emptyset$,
there is some $j\in \{1,2, \ldots, n\}$ such that $t_0^{-1}V'\cap t_j V\neq \emptyset$. Thus
$$(\{t_0t_jx\}\times t_0t_jV)\cap U'\times V'\neq \emptyset.$$
So we have \eqref{b1}.

Let $\{U_i\}_{i=1}^\infty$ be a base for $X\times X$. By \eqref{b1},
the set
$$A_i=\{y\in X: \ \text{there is some }\ t\in T  \ \text{such that}\ t(x,y)\in U_i\}$$
is a dense open subset of $X$.
It follows that
$$\{y\in X: \overline{\O}((x,y),T)=X\times X\}=\bigcap_{i=1}^\infty A_i$$
is residual in $X$.
\end{proof}

Our question is that does the relative version of theorem above hold?
\begin{ques}
Let $\pi: (X,T)\rightarrow (Y,T)$ be a non-trivial open proximal extension of minimal flows. For each $y\in Y$ and each $x\in \pi^{-1}(y)$, is the subset
$$\{x'\in \pi^{-1}(y): \overline{\O}\left((x,x'), T\right)=R_\pi \}$$
residual in $\pi^{-1}(y)$?
\end{ques}

\section{A dichotomy theorem on almost proximal extensions}\label{section-main}

In this section, we show that almost proximal extension has the following dichotomy theorem: any almost proximal extension of minimal flows is either almost finite to one, or almost all fibers contain an uncountable strongly scrambled subset.

\begin{lem}\label{lem-lift-distal}
Let $\pi: (X,T)\rightarrow (Y,T)$ be a distal extension of minimal flows.
For any strongly scrambled set $K$ of $Y$, there is a strongly scrambled set $K'$ of $X$ such that
$\pi|_{K'}: K'\rightarrow K$ is one to one.
\end{lem}

\begin{proof}
Let $K=\{y_\a\}_{\a\in \Lambda}$, where $\Lambda$ is an index set. Fix a point $y_0\in K$. By the assumption, $K\subseteq P[y_0].$ Then for each $\a\in \Lambda$, there is a minimal idempotent $v_\a\in J({\bf M})$ such that $y_\a=v_\a y_0$ . Thus
$$K=\{v_\a y_0\}_{\a\in \Lambda} \subseteq P[y_0].$$
Since $K$ is a strongly scrambled set, any pair $(y_\a,y_\b)$ is a recurrent point of $(Y\times Y, T)$ and there is some idempotent $u_{\a \b}\in J(S_T\setminus T)$ such that
$$u_{\a\b}(y_\a, y_\b)=(y_\a,y_\b).$$
Now choose a point $x_0\in \pi^{-1}(y_0)$, and let
$$K'=\{v_\a x_0\}_{\a\in \Lambda} .$$
For any $\a,\b\in \Lambda$, since $v_\a, v_\b$ are minimal idempotents, we have
$$v_\b (v_\a x_0, v_\b x_0)=(v_\b v_\a x_0, v_\b^2 x_0)=(v_\b x_0, v_\b x_0)$$
and hence $(v_\a x_0, v_\b x_0) \in P(X)$. By $u_{\a \b} y_\a = y_\a$, we have $\pi( u_{\a \b} v_\a x_0)=u_{\a \b} y_\a = y_\a= \pi( v_\a x_0)$. Since $\pi$ is distal, it follows that $u_{\a \b} v_\a x_0 =v_\a x_0$. Similarly, we have $u_{\a \b}v_\b x_0=v_\b x_0$. So
$$u_{\a \b} (v_\a x_0,v_\b x_0)=(v_\a x_0,v_\b x_0).$$
That is, $(v_\a x_0, v_\b x_0)$ is a recurrent point of $(X\times X,T)$. Thus $K'$ is a strongly scrambled subset of $X$. As $\pi$ is distal and each pair in $K'$ is proximal, $\pi|_{K'}: K'\rightarrow K$ is one to one. The proof is complete.
\end{proof}

\begin{thm}
Let $\pi: (X,T)\rightarrow (Y,T)$ be an almost proximal extension of minimal flows.
Then one of the following holds:
\begin{enumerate}
  \item $\pi$ is almost finite to one;
  \item there is a residual subset $Y_0\subseteq Y$ such that for every $y\in Y_0$, the fiber $\pi^{-1}(y)$ contains an uncountable strongly scrambled set.
\end{enumerate}
\end{thm}

\begin{proof}
We show that if $\pi$ is not almost finite to one, then there is a residual subset $Y_0\subseteq Y$ such that for every $y\in Y_0$, the fiber $\pi^{-1}(y)$ contains an uncountable strongly scrambled set.

Let $y_0\in Y$ and $u\in J(\M)$ such that $uy_0=y_0$ and let $u\pi^{-1}(y_0)=\{x_1,x_2,\ldots, x_n\}$ for some $n\in \N$. Let $z_0=(x_1,x_2,\ldots, x_n)\in X^n$. Then $z_0$ is a minimal point of $(X^n,T)$. Let
$$X^\#=\overline{\O}(z_0, T)=\{pz_0\in X^n: p\in \M\}\subseteq X^n.$$
Then $(X^\#,T)$ is a minimal flow. Let
$$\psi: X^\#\rightarrow X, \ pz_0\mapsto px_1, \quad \widetilde{\pi}=\pi\circ \psi: X^\#\rightarrow Y,  pz_0\mapsto py_0, \ \forall p\in \M . $$

\begin{equation*}
  \xymatrix@R=0.5cm{
                &         {X}^\# \ar[dl]_{\psi} \ar[dd]^{\widetilde{\pi}}    \\
  X \ar[dr]_{\pi}      \\
                &         Y          }
\end{equation*}

We divide the proof into several steps.

\medskip
\noindent {\bf Step 1. }\ {\em $\widetilde{\pi}$ is regular. }
\medskip

Let $p_1, p_2\in \M$ such that $\widetilde{\pi}(p_1z_0)=\widetilde{\pi}(p_2z_0)$ and $(p_1z_0, p_2z_0)$ is minimal. Let $v\in J(\M)$ such that $v(p_1z_0,p_2z_0)=(p_1z_0,p_2z_0)$. Since $\widetilde{\pi}(p_1 z_0)=\widetilde{\pi}(p_2z_0)$, we have that $p_1y_0 = p_2y_0$. Hence $up_1^{-1}p_2y_0=y_0$. Now let
$$\chi: X^\#\rightarrow X^\#, \ p z_0\mapsto p(up_1^{-1}p_2) z_0.$$
First, we verify that $\chi$ is well defined. Let $p,q\in \M$ such that $pz_0=qz_0$. Since $z_0=(x_1,x_2,\ldots, x_n)$, it follows that $px_j=qx_j, 1\le j\le n$. By $up_1^{-1}p_2\pi^{-1}(y_0)=u\pi^{-1}(y_0)$, $$\{up_1^{-1}p_2x_1,up_1^{-1}p_2 x_2,\ldots, up_1^{-1}p_2 x_n\}=\{x_1,x_2, \ldots, x_n\}.$$
Thus
$$p(up_1^{-1}p_2)x_j=q(up_1^{-1}p_2)x_j, \ 1\le j\le n,$$
i.e., $\chi(pz_0)=p(up_1^{-1}p_2)z_0=q(up_1^{-1}p_2)z_0=\chi(qz_0)$. That is, $\chi$ is well defined. As $\widetilde{\pi}(\chi(pz_0))=p(up_1^{-1}p_2)y_0=py_0=\widetilde{\pi}(pz_0)$, $\chi\in {\rm Aut}_{\widetilde{\pi}}(X^\#,T)$.

Finally, note that
$$\chi(p_1z_0)=\chi(vp_1z_0)=vp_1(up_1^{-1}p_2)z_0=vp_2z_0=p_2z_0.$$
Thus $\widetilde{\pi}$ is regular.

\medskip
\noindent {\bf Step 2. }\ {\em $\widetilde{\pi}$ is almost proximal and not almost finite to one. }
\medskip

Since $\pi$ is not almost finite to one, it is easy to see that $\widetilde{\pi}$ is also not almost finite to one. Next we show it is almost proximal.
We show that $|u\widetilde{\pi}^{-1}(y_0)|<\infty$. Let $pz_0\in u\widetilde{\pi}^{-1}(y_0)$. Then
$up(x_1,x_2,\ldots, x_n)=upz_0=pz_0=(px_1,px_2,\ldots, px_n)$. It follows that
$$\{px_1,px_2,\ldots, px_n\}=\{x_1,x_2,\ldots, x_n\}=u\pi^{-1}(y_0).$$
Thus $(px_1,px_2,\ldots, px_n)$ is a permutation of $(x_1,x_2,\ldots,x_n)$.
Hence $|u\widetilde{\pi}^{-1}(y_0)|\le |S_n|<\infty$, where $S_n$ is the symmetric group on $\{x_1,x_2,\ldots, x_n\}$. That is, $\widetilde{\pi}$ is almost proximal.

\medskip
\noindent {\bf Step 3. }\ {\em There is a residual subset $Y_0\subseteq Y$ such that for every $y\in Y_0$, the fiber $\pi^{-1}(y)$ contains an uncountable strongly scrambled set. }
\medskip

Since $\widetilde{\pi}$ is regular and almost proximal, by Theorem \ref{thm-reg}, there is a finite to one equicontinuous extension $\phi: X^\#\rightarrow Y^\#$ and a proximal extension $\pi^\#: Y^\#\rightarrow Y$ such that $\widetilde{\pi}=\pi^\#\circ \phi$.

\begin{equation*}
  \xymatrix@R=0.5cm{
                &         {X}^\# \ar[dl]_{\psi} \ar[dd]^{\widetilde{\pi}} \ar[dr]^{\phi} &    \\
  X \ar[dr]_{\pi}     &       &  Y^\#  \ar[dl]^{\pi^\#}   \\
                &         Y         &        }
\end{equation*}

Since $\widetilde{\pi}$ is not almost finite to one, $\pi^\#$ is proximal but not almost one to one. By Theorem \ref{thm-prox}, there is a residual subset $Y_0\subseteq Y$ such that for every $y\in Y_0$, the fiber $(\pi^\#)^{-1}(y)$ contains an countable strongly scrambled set $K'_y$. Since $\phi$ is distal, by Lemma \ref{lem-lift-distal}, there is an countable strongly scrambled set $K^\#_y$ such that $\phi(K^\#_y)=K'_y$.
Note that $K^\#_y\subseteq X^\#\subseteq X^n$. Let
$K_j=\pi_j(K^\#_y), 1\le j\le n$, where $\pi_j$ is the projection from $X^n$ to $j$-th coordinate.
Since $$K^\#_y\subseteq K_1\times K_2\times \ldots \times K_n,$$
there is some $j_0\in \{1,2,\ldots, n\}$ such that $|K_{j_0}|$ is uncountable.
Since $K^\#$ is a strongly scrambled set of $X^\#$, $K_{j_0}$ is a strongly scrambled set
of $X$. As the diagram is commutative, it is easy to check that $K_{j_0}\subseteq \pi^{-1}(y)$.
The proof is complete.
\end{proof}



\section{Examples}\label{section-example}

Since there is no non-trivial proximal minimal flow under abelian group action \cite[Theorem 3.4]{G76}, it is not easy to construct a proximal but not almost one to one extension of minimal flows. In fact this was a question by Furstenberg several years ago. Using category method, Glasner and Weiss constructed the first this kind of extensions of minimal flows and gave a positive answer to this question. In this section, using methods in \cite{CM06}, we give explicit examples of proximal but not almost one to one extensions. And examples constructed are uniformly rigid.

In this section, firstly we briefly introduce Glasner and Weiss' results, which we will use them later. For Glasner and Weiss' methods, refer to \cite{GW79} for more details. Then we show how to give explicit examples of almost proximal but not almost finite to one extensions.

\subsection{Glasner and Weiss' results}\label{GW-exam}\
\medskip

Let $(Y,f)$ be a minimal discrete flow and $Z$ be a compact metric space.
Let $H(Z,Z)$ be the space of all homeomorphisms of $Z$ equipped with the metric
$$\rho_Z(g,h)=\sup_{z\in Z}d_Z(g(z),h(z))+ \sup_{z\in Z}d_Z(g^{-1}(z),h^{-1}(z)).$$
With this metric $H(Z,Z)$ is a complete metric space and a topological group. Let $X=Y\times Z$. Let $H_s(X,X)$ be subspace of $H(X,X)$ which consists of homeomorphisms which fixes all subspace of $X$ of the form $\{y\}\times Z, y\in Y$. Such a homeomorphism $G$ is determined by a continuous map $\a: Y\rightarrow H(Z,Z)$, by $G(y,z)=(y, \a(y)(z))$. Put
$$\mathcal{S}(f)=\{G^{-1}\circ f\circ G: G\in H_s(X,X)\}.$$

(Here $f$ is identified with $f\times \id_Z,$ where $\id_Z$ is the identity map on $Z.$)

If $\G$ is a subgroup of $H(Z,Z)$, let $\G_s\subseteq H_s(X,X)$ be the subgroup of those elements of $H_s(X,X)$ which come from cocycles of $\G$. That is,
$$\G_s=\{G\in H_s(X,X): G(y,z)=(y,\a(y)(z)), \a\in C(Y,\G)\}. $$
Let
$$\mathcal{S}_\G(f)=\{G^{-1}\circ f\circ G: G\in \G_s\}.$$

\begin{thm}\cite[Theorem 1.]{GW79}\label{thm-GW1}
Let $\G$ be a pathwise connected  subgroup of $H(Z,Z)$ such that $(Z,\G)$ is a minimal flow. If $(Y,f)$ is minimal, then for a residual subset $\mathcal {R}\subseteq \overline{\mathcal{S}_\G(f)}$, $(X,F)$ is a minimal flow for every $F\in \mathcal{R}$.
\end{thm}

\begin{thm}\cite[Theorem 3.]{GW79}\label{thm-GW2}
Let $\G$ be a pathwise connected  subgroup of $H(Z,Z)$ with the following property: for every pair of points $z_1,z_2\in Z$ there exist neighbourhoods $U$ and $V$ of $z_1$ and $z_2$ respectively, such that for every $\ep>0$ there exists $h\in \G$ with ${\rm diam}\ (h(U\cup V))<\ep$. Then for a residual subset $\mathcal {R}\subseteq \overline{\mathcal{S}_\G(f)}$, $(X,F)$ is a proximal extension of $(Y,f)$ for every $F \in \mathcal{R}$.
\end{thm}

According to \cite{GW79}, if we choose $Z={\bf P}^n, n\ge 1$ to be the projective $n$-space, and let $\G$ be pathwise connected component of $\id_Z$ in $H(Z,Z)$, then $\G$ satisfies the conditions of Theorem \ref{thm-GW1} and Theorem \ref{thm-GW2}. Thus for an arbitrary minimal infinite flow $(Y,f)$, there are many minimal homeomorphisms of $Y\times Z$ which are proximal but not almost one to one extensions of $(Y,f)$.

\subsection{Skew-product}\
\medskip

Let $(Y,f)$ be a discrete flow and $(Z,d_Z)$ a compact metric space. Denote the set of all the homeomorphisms of $Z$ to $Z$ with $H(Z,Z)$. For $\varphi _1, \varphi _2\in H(Z,Z)$, set
$$D_Z(\varphi _1,\varphi _2)=\sup_{z\in Z}d_Z(\varphi _1(z),\varphi _2(z)).$$ 

Assume $X=Y\times Z$ and let $\rho_X $ denote the max-metric on the product space $Y\times Z$,
$$\rho_X\left((y_1,z_1),(y_2,z_2)\right)=\max\{d_Y(y_1,y_2),d_Z(z_1,z_2)\},$$
where $d_Y$ is the metric of $Y$.

A continuous map $\sigma: Y\rightarrow H(Z,Z)$ is called a {\em cocycle}. By a given cocycle $\sigma$, one can define
\[f_\sigma : X\rightarrow X, \ (y,z)\mapsto \left(f(y), \sigma (y)(z)\right), \ \forall (y,z)\in X.\]
The new flow $(X,f_\sigma)$ is called a {\em skew-product flow}.

\medskip

For $(y,z)\in X$, set $f^n_\sigma (y,z)= (f^n (y), \sigma_n(y)(z))$.
Then
$$\sigma_n(y)=
\left\{
             \begin{array}{ll}
               \sigma(f^{n-1}y)\cdots \sigma(fy)\sigma(y), & n\ge 1 \hbox{;} \\
               \id_Z, & n=0 \hbox{;}\\
               \sigma(f^ny)^{-1}\cdots \sigma(f^{-2}y)^{-1}\sigma(f^{-1}y)^{-1}, & n<0 \hbox{.}
             \end{array}
           \right.
$$

\subsection{A cocycle}\
\medskip

Let $Y=Z(2)=\{0,1\}^\N$ with the metric
$$d(\a,\b) = \frac1{\min \left\{i\in \N:
\a_i\neq \b_i\right\}}, \a=(\a_i)_{i=1}^\infty, \b=(\b_i)_{i=1}^\infty \in
Z(2).$$
The map $$\tau : Z(2)\rightarrow Z(2)$$
is defined as follows: for
every $\a \in Z(2)$, $\tau (\a)=\a + 10000\dots$, where the addition
is modulo $2$ from the left to right. Obviously, $\tau$ is continuous.
Moreover, it can be shown that $\tau$ is invertible and $(Z(2),\tau)$ is an equicontinuous minimal
flow. The flow $(Z(2),\tau)$ is called {\em adding machine} or {\em
odometer}. Similarly, one can define $(Z(k)=\{0,1,\ldots, k-1\}^\N, \tau)$.

In the sequel an increasing sequence $\{n_i\}_{i=1}^\infty \subseteq \N$ is fixed. For any $\a\in Z(2)$, $\a$ can be written as
\begin{equation}\label{s1}
    \a=\a^1 \a^2 \a^3 \dots, \ \text{where}\ \a^k \ \text{is a block of $n_k$ digits of $\a$}.
\end{equation}
Let $e: \bigcup_{n=1}^\infty \{0,1\}^n\rightarrow \Z_+$ be the {\em
evaluation function}: for $x=x_1x_2\ldots x_n \in \{0,1\}^n$,
$$e(x)= x_1+2x_2+2^2x_3+\ldots + 2^{n-1}x_n$$
and let
$$\underline n = e^{-1}(n), \ \forall n \in \Z_+ .$$

\medskip

Let $Z=\bbs^1$ and
\begin{equation}\label{}
  \Phi=\{\varphi_k^j: 0\le j\le
2^{n_k}-2\}_{k=1}^{\infty}\subseteq H(\bbs^1,\bbs^1).
\end{equation}
We use $\Phi$ to construct a cocycle $\sigma: Z(2)\rightarrow H(\bbs^1,\bbs^1)$:
\begin{equation}\label{}
\sigma(\a)=
\left\{
  \begin{array}{ll}
    \varphi^{e(\a^k)}_{k}, & \hbox{$\a\neq 1^\infty$ and $\a^k$ is the first block in
(\ref{s1}) containing at least one zero digit;} \\
    \id_{\bbs^1}, & \hbox{$\a =1^\infty$.}
  \end{array}
\right.
\end{equation}


\begin{lem}
If $\Phi$ satisfies the following condition
\begin{equation}\label{C1}
  \displaystyle \lim_{k\to \infty} \max_{0\le j\le 2^{n_k}-2} D_{\bbs^1}(\varphi^{j}_k, \id_{\bbs^1} ) = 0,\tag{C1}
\end{equation}
then $\sigma: Z(2)\rightarrow H(\bbs^1,\bbs^1)$ is continuous and hence it is a cocycle.
\end{lem}

\begin{proof}
Let $\b_n, \a \in Z(2)$ such that $\b_n \to \a, n\to\infty$. We show that $\lim_{n\to \infty} D_{\bbs^1}(\sigma(\b_n), \sigma(\a))=0$.
	
If $\a \ne 1^\infty$, then one has $\sigma(\b _n)=\sigma(\a)$ when $n$ is large enough and $D_{\bbs^1}(\sigma(\b_n), \sigma(\a))=0$. Now assume that $\a =1^\infty$. For $\b_n$, we have that
$ \b_n=\b_n^1 \b_n^2 \b_n^3 \ldots$, where $\b_n^k$ is a block of $n_k$ digits of $\b_n$.
Thus $\sigma(\b_n)=\phi_k^{e(\b_n^k)}$, where $\b_n^k$ is the first block containing at least one zero digit. Since $\b_n\to\a, n\to\infty$ and $\lim_{k\to \infty} \max_{0\le j\le 2^{n_k}-2} D_{\bbs^1}(\varphi^{j}_k, \id_{\bbs^1} ) = 0$, we have that
$$\lim_{n\to \infty} D_{\bbs^1}(\sigma(\b_n), \id_{\bbs^1})=0.$$
That is, $\sigma$ is continuous.
\end{proof}

\medskip
Thus when \eqref{C1} is satisfied, $\sigma$ is a cocycle. Hence we have a skew product flow $(X=Z(2)\times \bbs^1, F_\Phi)$:
\begin{equation}\label{}
  F_\Phi\triangleq \tau_\sigma: X\rightarrow X, \
F_\Phi(\a, y)=(\tau(\a), \sigma(\a) y)=
\left\{
  \begin{array}{ll}
    (\tau(\a), \varphi^{e(\a^k)}_{k}(y)), & \hbox{$\a\neq 1^\infty$;} \\
    (\underline{0}, y), & \hbox{$\a =1^\infty$,}
  \end{array}
\right.
\end{equation}
where $\a^k$ is the first block in
(\ref{s1}) containing at least one zero digit when $\a\neq 1^\infty$.


\subsection{The form of $F_\Phi ^n$}\
\medskip

Let $\a=\a^1\a^2\dots \in Z(2)$ as in (\ref{s1}) and $z_0\in \bbs^1$. Let
$$F^n_\Phi (\a ,z_0)=(\tau ^n\a,z_n).$$ We need to know the formula of $z_n$ for some special $n$.

For simplicity, we always assume the following conditions hold:
\begin{equation}\label{C2}
   \varphi^0_k = \id_{\bbs^1},\ \forall k\in \N,\tag{C2}
\end{equation}
and
\begin{equation}\label{C3}
  \Psi_k= \varphi_k^{2^{n_k}-2} \circ \varphi_k^{2^{n_k}-3} \circ \cdots \circ \varphi^1_k\circ \varphi^0_k=\id_{\bbs^1},\  \forall k\in\N.\tag{C3}
\end{equation}


First we have a formula for $\a= \underline 0$. By an easy induction, we have

\begin{lem}\label{lem-6.4}
Let $(\underline 0 ,y_0) \in Z(2)\times \bbs^1$ and $(\underline n , y_n)=F_\Phi ^n (\underline 0 , y_0)$.
Let $m_k=2^{n_1+n_2+\ldots +n_k}$ for all $k\in \N$. Then
\begin{equation}\label{eq1}
y_{m_k}=\varphi^0_{k+1}\circ\psi_k(y_0)=y_0,
\end{equation}	
and
\begin{equation}\label{2}
y_{s\cdot m_k}=\varphi^{s-1}_{k+1}\circ \varphi^{s-2}_{k+1}\circ \dots \circ \varphi^1_{k+1}\circ \varphi ^0_{k+1}(y_0), 1\le s <2^{n_{k+1}}.
\end{equation}
\end{lem}

Now we see the general case. Let $\a=\a^1\a^2\dots \in Z(2)$ as in
(\ref{s1}) and $z_0\in \bbs^1$. We denote $F^n_\Phi (\a ,z_0)=(\tau ^n\a,z_n).$
Let
\begin{equation}\label{s4}
\psi_k ^+=\varphi^{e(\a^k)-1}_k\circ\varphi^{e(\a^k)-2}_k \circ \dots \circ \varphi^{1}_k\circ \varphi^{0}_k,\ \forall k\in \N,
\end{equation}
and
\begin{equation}\label{s5}
    \psi _k ^-=\varphi^{2^{n_k}-2}_k\circ\varphi^{2^{n_k}-3}_k\circ
\dots \circ \varphi^{e(\a^k)+1}_k\circ \varphi^{e(\a^k)}_k, \ \forall k\in \N.
\end{equation}
By the assumption (\ref{C3}), one has that $(\psi_k ^+)^{-1}=\psi ^-_k.$

\medskip

By an easy induction, we have

\begin{lem}\label{lem-6.5}
Let $\a=\a^1\a^2\dots \in Z(2)$ as in (\ref{s1}) and $z_0\in \bbs^1$. Let $F^n_\Phi (\a ,z_0)=(\tau ^n\a,z_n).$ Let $m_k=2^{n_1+n_2+\ldots +n_k}$ for all $k\in \N$. Then
\begin{equation}\label{s6}
F_\Phi ^{m_k -e(\a^1\a^2\dots\a^k)}(\a,z_0)=(0^{n_1+\ldots + n_k}\tau(\a^{k+1}\a^{k+2}\cdots),z_{m_k -e(\a^1\a^2\dots\a^k)}),
\end{equation}
where $z_{m_k-e(\a^1\a^2\dots \a^k)}=\varphi^{e(\a^{k+1})}_{k+1}\circ
\psi^-_k\circ \psi^-_{k-1}\circ\dots \circ \psi^-_1 (z_0). $
And
\begin{equation}\label{1}
F_\Phi^{m_k}(\a,z_0)=(\a^1\a^2\dots\a^k\tau(\a^{k+1}\a^{k+2}\dots),z_{m_k}),
\end{equation}
where $z_{m_k}=\psi^+_1\circ \psi^+_2\circ \dots \circ \psi^+_k \circ \varphi ^{e(\a^{k+1})}_{k+1}\circ \psi^-_k\circ \dots \circ \psi^-_1 (z_0).$
\end{lem}

\begin{rem}
(i) To simply the calculation, if we require that
	\begin{equation*}
	\varphi^1_k =\varphi ^2_k =\dots =\varphi^{2^{n_k -1}-1}_k\triangleq\varphi_k
	\end{equation*}
	\begin{equation*}
	\varphi^{2^{n_k -1}}_k=\varphi^{2^{n_k -1}+1}_k=\dots =\varphi^{2^{n_k}-2}\triangleq \varphi^-_k=(\varphi_k)^{-1},
	\end{equation*}
then $\psi^+_k=(\varphi_k)^{c_k},$ $\psi^-_k=(\varphi^-_k)^{c_k},$ where $c_k=2^{n_k -1}-1-|e(\a^k)-2^{n_k-1}|.$	 Thus (\ref{1}) will be
		\begin{equation}z_{m_k}=(\varphi_1)^{c_1}\circ (\varphi_2)^{c_2}\circ\dots \circ (\varphi_k)^{c_k}\circ \varphi ^{e(\a^{k+1})}_{k+1}\circ (\varphi ^-_k)^{c_k}\circ \dots \circ (\varphi ^-_1)^{c_1}(z_0).	
		\end{equation}

\medskip
\noindent (ii) By (\ref{2}), (\ref{s6}), for $j<k,$ $1\leq s<2^{n_{j+1}}$, we have
	\begin{equation}
	F^{m_k-e(\a^1\a^2\dots \a^k)+s\cdot m_j}_\Phi (\a ,z_0)=(\underline s\tau(\a^{k+1}\a ^{k+2}\dots), z_{m_k -e(\a^1\a^2\dots \a^k)+s\cdot m_j}),
	\end{equation}
	where $z_{m_k -e(\a^1\a^2\dots \a^k)+s\cdot m_j}=\varphi ^{s-1}_{j+1}\circ\varphi ^{s-2}_{j+1}\circ \dots\circ \varphi ^{1}_{j+1} \circ\varphi ^{0}_{j+1}(z_{m_k -e(\a^1\a^2\dots \a^k)}).$

\end{rem}


\subsection{Uniform rigidity}\
\medskip

A discrete flow $(X,F)$ is {\em rigid} if there exists an increasing sequence $\{n_i\}_{i\in \N}$ in $\N$ such that $F^{n_i}x$ converges to $x$ as $i$ goes to infinity for every $x\in X$ ( i.e., $F^{n_i}$ converges pointwisely to the identity map). A flow $(X, F)$ is {\em uniformly rigid} if there exists an increasing sequence $\{n_i\}_{i\in \N}$ in $\N$ such that $\sup\limits_{x\in X} d(x,F^{n_i}x) \to 0$ as $i$ goes to infinity ( i.e., $F^{n_i}$ converges uniformly to the identity map). 
Refer to \cite{GM} for more information about topological rigidity.

\begin{lem}\label{easy1}
Let $\F$ be a finite subset of $H(X,X)$, where $(X,\rho_X)$ is a compact metric space. Then for any
$\ep>0$ there is a $\d>0$ such that for any continuous map $h: X\rightarrow X$
with $D_X(h,\id_X)<\d$, we have
$$D_X(\psi\circ h \circ \psi^{-1}
, \id_X)<\ep ,  \ \forall \ \psi \in \F.$$
\end{lem}

\begin{proof}
For any $\ep>0$ there exists $\d >0$ such that whenever $\rho_X(x_1,x_2)<\d,$  \[\rho_X(\psi^{-1}(x_1),\psi^{-1}(x_2)) <\ep, \ \forall \psi\in \F.\]
Then for continuous map $h: X\rightarrow X$
with $D(h,\id_X)<\d$, we have that
$$\rho_X(h(\psi(x)),\psi(x))<\d, \forall x\in X, \forall \psi\in \F.$$
It follows that \[\rho_X(\psi ^{-1}\circ h\circ \psi (x),x)<\ep,\  \forall x\in X, \forall \psi\in \F. \] That is,
\[D_X(\psi\circ h \circ \psi^{-1}
	, \id_X)<\ep , \ \forall \psi\in \F.\] The proof is complete.
\end{proof}

\begin{prop}\label{prop-rigid}
Suppose that $\Phi$ satisfies \eqref{C1}, \eqref{C2} and \eqref{C3}. Then the skew product flow     $(Z(2)\times \bbs^1,F_\Phi)$ is uniformly rigid.
\end{prop}

\begin{proof}
By the formula \eqref{1}, for any $(\a, z_0)\in X=Z(2)\times \bbs^1 ,$ we get
	\[F^{m_k}_\Phi (\a,z_0)=(\a^1\a^2\dots \a^k\tau(\a^{k+1}\a^{k+2}\dots), \psi(\a^1\a^2\dots \a^k)\varphi ^{e(\a^{k+1})} _{k+1}\psi^{-1}(\a^1\a^2\dots \a^k)(z_0)),\]
where $\psi(\a^1\a^2\dots \a^k)=\psi^+_1\circ\psi^+_2\circ\dots \circ \psi^+_k .$
By the assumption \eqref{C1} that $$\lim_{k\to \infty} \max_{0\le j\le 2^{n_k}-2} D_{\bbs^1}(\varphi^{j}_k, \id_{\bbs^1}) = 0,$$ and Lemma \ref{easy1}, we have that
\[\lim_{k\to \infty} D_X(F_\Phi^{m_k},\id_X) =0 ,\]
that is, $(Z(2)\times \bbs^1,F_\Phi)$ is uniformly rigid.
\end{proof}


\subsection{Minimality}\
\medskip

Now we specify the homeomorphisms $\{\varphi_{2k-1}^j: 1\le j\le
2^{n_{2k-1}}-2\}_{k=1}^\infty$ to make the
flow $(Z(2)\times \bbs^1,F_\Phi)$ minimal. Let
\begin{equation}\label{C4}
\varphi_{2k-1}^j(z)=
\left\{
             \begin{array}{ll}
               e^{i\theta_k}z, & 1< j \le 2^{n_{2k-1}-1}-1 \hbox{;} \\
               e^{-i\theta_k}z, & 2^{n_{2k-1}-1}-1<j\le 2^{n_{2k-1}}-2 \hbox{.}
             \end{array}
           \right.\tag{C4}
\end{equation}
where $\theta_k = \cfrac {2\pi}{2^{n_{2k-1}-1}-1}$. That is,
$\varphi^{j}_{2k-1}$ is a rotation of $\bbs^1$ with angle $\theta_k$ in the
anti-clockwise direction if $1< j \le 2^{n_{2k-1}-1}-1$ and in the
opposite direction otherwise. We remark that this construction satisfies the assumption \eqref{C3}, i.e., $\psi _{2k-1} =\id.$

\begin{prop}\label{prop-min}
Suppose that $\Phi$ satisfies \eqref{C1}, \eqref{C2}, \eqref{C3} and \eqref{C4}. Then the skew product flow     $(Z(2)\times \bbs^1,F_\Phi)$ is uniformly rigid and minimal.
\end{prop}

\begin{proof}
Let $\a\in Z(2)$ and $\bbs^1_{\a}={\a}\times \bbs^1.$
First we show that for any $y_0\in \bbs^1$,
\[\overline{\O}((\underline 0,y_0),F_\Phi)=Z(2)\times \bbs^1.\]
By (\ref{2}), for $1\le s<2^{n_{2k-1}}$, one has that
\[F^{s\cdot m_{2k-2}}_\Phi(\underline 0,y_0)=(0^{n_1+n_2+\dots n_{2k-2}}\underline s0^\infty ,y_{s\cdot m_{2k-2}}),\] where $y_{s\cdot m_{2k-2}}=\varphi^{s-1}_{2k-1}\circ \varphi^{s-2}_{2k-1}\circ \dots \circ \varphi^1_{2k-1}\circ \varphi ^0_{2k-1}(y_0).$
As a result, $\{y_{s\cdot m_{2k-2}}:\ 1\le s \le 2^{n_{2k-1}-1}-1\}$ is $\theta _k =\cfrac {2\pi}{2^{n_{2k-1}-1}-1}-$dense in $\bbs^1.$ Moreover, for any $y\in \bbs^1,$ there exists $1\le s_k\le 2^{n_{2k-1}-1}-1$ such that \[\lim_{k\to \infty}y_{s_k\cdot m_{2k-2}}=y.\] Thus \[F^{s_k\cdot m_{2k-2}}_\Phi(\underline 0 ,y_0)\to (\underline 0,y), k\to\infty.\] Thus $\bbs^1_{\underline 0} \subseteq \overline{\O}((\underline 0,y_0),F_\Phi).$
Since $F^n_\Phi(\bbs^1_{\underline 0})=\bbs^1_{\underline n},$ we have \[Z(2)\times \bbs^1=\overline{\bigcup_{n\in \Z _{+}} \bbs^1_{\underline n}} \subseteq  \overline{\O}((\underline 0,y_0),F_\Phi).\]
Thus \[\overline{\O}((\underline 0,y_0),F_\Phi)=Z(2)\times \bbs^1, \ \forall y_0\in \bbs^1.\]

Now we show that for any $(\a ,z_0 )\in Z(2)\times \bbs^1,$ one has \[\overline{\O} ((\a, z_0),F_\Phi)=Z(2)\times \bbs^1.\] As $(Z(2),\tau )$ is minimal, there exists some $\{p_k\}$ such that \[\lim _{k\to \infty}\tau ^{p_k} (\a)=\underline 0.\]
Without loss of generality, we may assume that
$$\lim_{k\to \infty} F^{p_{k}}_\Phi(\a,z_0 ) =(\underline 0, y_0)$$ for some $y_0 \in \bbs^1.$
Thus $Z(2)\times \bbs^1=\overline{\O}( (\underline 0, y_0),F_\Phi) \subseteq \overline{\O} ((\a , z_0),F_\Phi).$ Thus
$$\O ((\a, z_0),F_\Phi)=Z(2)\times \bbs^1.$$
Hence $(Z(2)\times \bbs^1,F_\Phi)$ is minimal. Moreover, by Proposition \ref{prop-rigid}, it is uniformly rigid. The proof is complete.
\end{proof}


\subsection{Proximality}\label{1-a.p}\
\medskip

In this subsection, we specify the homeomorphisms $\{\varphi_{2k}^j: 1\le j\le
2^{n_{2k}}-2\}_{k=1}^\infty$ to make the
extension $\pi: (Z(2)\times \bbs^1, F_\Phi )\to (Z(2),\tau), \ (\a , z_0)\mapsto \a$ to be proximal.
Let
\begin{equation}\label{C5}
  \varphi_{2k}^j( e^{i{2\pi x }})=
\left\{
\begin{array}{ll}
\exp\big(i{2\pi  x^{t_k} }\big), & 1< j \le 2^{n_{2k}-1}-1 \hbox{;} \\
\exp\big( {i2\pi  x^{1/{t_k}}}\big), & 2^{n_{2k}-1}-1<j\le 2^{n_{2k}}-2 \hbox{,}
\end{array}
\right.\tag{C5}
\end{equation}
where $\{t_k\}$ is a decreasing sequence tending to $1.$
Notice that the construction of $\varphi _{2k} ^j$  also satisfies the convention \eqref{C3}, i.e., $\psi _{2k}=\id.$

\medskip
To verify that $\pi$ is proximal, we begin with the following lemma.

\begin{lem}\label{proximal}
	Let $\pi: (X,T)\rightarrow (Y,T)$ be an extension of minimal flows. If there exists $y_0 \in Y$ such that for any $x_1 , x_2 \in \pi ^{-1} (y_0) ,$ $(x_1,x_2)\in P(X,T),$ then $\pi$ is proximal.	
\end{lem}

\begin{proof}
	Let $(x_1,x_2)\in R_\pi$ and $\pi (x_1)=\pi (x_2)=y_1 .$ Since $(Y,T)$ is minimal, we can find $p\in E(X,T) $ such that $$py_1=y_0 .$$
	So $px_1,px_2\in \pi^{-1}(y_0).$ According to the assumption, there exists $q\in E(X,T)$ such that $$qpx_1=qpx_2.$$
	It follows that $$(x_1,x_2)\in P(X,T).$$
Thus $\pi$ is proximal.
\end{proof}

\begin{prop}\label{prop-proximal}
Suppose that $\Phi$ satisfies \eqref{C2}, \eqref{C4} and \eqref{C5}. Then the
extension $\pi: (Z(2)\times \bbs^1, F_\Phi )\to (Z(2),\tau), \ (\a , z)\mapsto \a$ is a proximal extension of uniformly rigid minimal flows.
\end{prop}

\begin{proof}
First notice that \eqref{C4} and \eqref{C5} imply \eqref{C1} and \eqref{C3}.  It follows that $(Z(2)\times \bbs^1, F_\Phi )$ is uniformly rigid and minimal by Proposition \ref{prop-min}.

By Lemma \ref{proximal}, it suffices to show that for arbitrary $z_0^{(1)}=e^{i{2\pi x_1 }}, z_0^{(2)} =e^{i{2\pi x_2 }} \in \bbs^1,$ one has \[((\underline 0 ,z_0^{(1)}),(\underline 0 ,z_0^{(2)}))\in P(Z(2)\times \bbs^1,F_\Phi).\]
Let $d_{\bbs^1}$ be the metric of $\bbs^1$. Let $F^{m}_\Phi((\underline 0 ,z_0^{(j)}))=(\underline 0 ,z_m^{(j)})$ for $j=1,2$ and $m\in \N$. Let $m_k=2^{n_1+n_2+\ldots +n_k}$ for all $k\in \N$. Then by Lemma \ref{lem-6.4},
\begin{equation*}
  \begin{split}
    & \quad \lim _{k\to \infty}d_{\bbs^1}\left(z^{(1)}_{2^{n_{2k}-1}\cdot m_{2k-1}},z^{(2)}_{2^{n_{2k}-1}\cdot m_{2k-1}}\right)\\
& =\lim _{k\to \infty} d_{\bbs^1}\left (\exp\Big(i{2\pi  x_1^{t_k^{(2^{n_{2k}-1}-1)}} }\Big), \exp\Big(i{2\pi  x_2^{t_k^{(2^{n_{2k}-1}-1)}} }\Big)\right)\\
&=d_{\bbs^1}(1,1)=0.
   \end{split}
\end{equation*}
In particular,
$$\lim_{k\to\infty} \rho_{X}\left((F_\Phi\times F_\Phi)^{2^{n_{2k}-1}\cdot m_{2k-1}}\left((\underline 0 ,z_0^{(1)}),(\underline 0 ,z_0^{(2)})\right)\right)=0.$$
It follows that $((\underline 0 ,z_0^{(1)}),(\underline 0 ,z_0^{(2)}))\in P(Z(2)\times \bbs^1,F_\Phi)$.
\end{proof}

\subsection{Proximal extensions and $n$-weakly mixing extensions}\
\medskip

In Section \ref{subsection-wm}, we mentioned that the fact that $\pi: (X,T)\rightarrow (Y,T)$ is weakly mixing can not imply that $\pi$ is totally weakly mixing. In this subsection, we give such examples.
The main result of this subsection is as follows:

\begin{thm}\label{thm-exam-wm}
\begin{enumerate}
  \item There are proximal extensions of discrete minimal flows which are weakly mixing but not $3$-weakly mixing.
  \item There are proximal extensions of discrete minimal flows which are totally weakly mixing.
\end{enumerate}
\end{thm}

\begin{proof}[Proof of (1) of Theorem \ref{thm-exam-wm}]
Suppose that $\Phi$ satisfies \eqref{C2}, \eqref{C4} and \eqref{C5}. We show that the
extension $$\pi: (Z(2)\times \bbs^1, F_\Phi )\to (Z(2),\tau), \ (\a , z)\mapsto \a$$ is weakly mixing but not $3$-weakly mixing. Since $\pi$ is open proximal by Proposition \ref{prop-proximal}, it is  weakly mixing by Theorem \ref{pwm}.
Now we show that $\pi$ is not $3$-weakly mixing, i.e., $(R_\pi ^3,F_\Phi^{(3)})$ is not transitive.

We may regard $R_\pi ^3$ as $Z(2)\times \bbs^1 \times \bbs^1\times \bbs^1.$ Suppose $(R_\pi ^3,F_\Phi^{(3)})$ is transitive, and let $(\a,x_1,x_2,x_3)\in Z(2)\times \bbs^1 \times \bbs^1\times \bbs^1$ be a transitive point. Let $y_1, y_2,y_3$ be distinct points of $\bbs^1$ and $\b\in Z(2)$. Then $(\b, y_1,y_2,y_3)$ in the orbit closure of the transitive point $(\a,x_1,x_2,x_3)$. By the construction of $\Phi$ (\eqref{C4} and \eqref{C5}), $F_\Phi$ preserves orientation. It follows that $(\b, y_1,y_3,y_2)$ can not be in the orbit closure of the transitive point $(\a,x_1,x_2,x_3)$.
Thus $(R_\pi ^3,F_\Phi^{(3)}) $ is not transitive.
\end{proof}

For the proof of (2) of Theorem \ref{thm-exam-wm}, we need some preparations. We use the notations in Subsection \ref{GW-exam}.

\begin{thm}\cite[Theorem 4.]{GW79}\label{thm-GW3}
Let $\G$ be a pathwise connected  subgroup of $H(Z,Z)$ such that $(Z,\G)$ is a weakly mixing flow. Then for a residual subset $\mathcal {R}\subseteq \overline{\mathcal{S}_\G(f)}$, $(X,F)$ is a weakly mixing extension of $(Y,f)$ for every $F\in \mathcal{R}$.
\end{thm}

By the same proof of \cite[Theorem 4.]{GW79}, we can show the following: If $\G$ is a pathwise connected  subgroup of $H(Z,Z)$ such that $(Z^n,\G)$ is a transitive flow, then for a residual subset $\mathcal {R}_n \subseteq \overline{\mathcal{S}_\G(f)}$, $(X,F)$ is a $n$-weakly mixing extension of $(Y,f)$ for every $F\in \mathcal{R}_n$.
Let ${\mathcal R}=\bigcap_{n=2}^\infty {\mathcal R}_n$. Then we have the following result which slightly generalizes Theorem \ref{thm-GW3}.

\begin{thm}\label{thm-GW4}
Let $\G$ be a pathwise connected  subgroup of $H(Z,Z)$ such that $(Z,\G)$ is a flow such that $(Z^n, \G)$ is transitive for all $n\ge 2$. Then for a residual subset $\mathcal {R}\subseteq \overline{\mathcal{S}_\G(f)}$, $(X,F)$ is a totally weakly mixing extension of $(Y,f)$ for every $F\in \mathcal{R}$.
\end{thm}

According to \cite{GW79}, if we choose $Z={\bf P}^n, n\ge 2$ to be the projective $n$-space, and let $\G$ be pathwise connected component of $\id_Z$ in $H(Z,Z)$, then $\G$ satisfies the conditions of Theorem \ref{thm-GW1}, Theorem \ref{thm-GW2} and Theorem \ref{thm-GW4}. Thus for an arbitrary minimal infinite flow $(Y,f)$, there are many minimal homeomorphisms of $Y\times Z$ which are proximal and totally weakly mixing extensions of $(Y,f)$. And we have proved (2) of Theorem \ref{thm-exam-wm}.

\begin{rem}
Note that the examples in Theorem \ref{thm-exam-wm}-(1) are uniformly rigid by Proposition \ref{prop-proximal}. By the proof of Proposition 6.5 of \cite{GM}, one may also require that the examples in Theorem \ref{thm-exam-wm}-(2) are uniformly rigid.
\end{rem}

\subsection{Almost proximal}\
\medskip

In this subsection, we modify the construction of homeomorphisms $\{\varphi_{2k}^j: 1\le j\le
2^{n_{2k}}-2\}_{k=1}^\infty$ to make the
extension $$\pi: (Z(2)\times \bbs^1, F_\Phi )\to (Z(2),\tau), \ (\a , z_0)\mapsto \a$$ almost proximal.

Let $n\in \N $ be a fixed number. Let $\omega=e^{i\frac{2\pi}{n}}$. We always write intervals on $\bbs^1$ anticlockwise, so $[z_1,z_2]$ denotes the anticlockwise closed interval beginning at $z_1$ and ending at $z_2$.
$$S_q=[\w^{q-1}, \w^q]=[e^{i\frac{2\pi}{n} (q-1) }, e^{i\frac{2\pi}{n} q }], \ 1\le q\le n.$$
Then
$$\bbs^1=S_1\cup S_2\cup\ldots \cup S_n.$$
Thus $\bbs^1$ is divided into $n$ closed intervals equally.
For each $q\in \{1,2,\ldots,n\}$, we may regard $S_q$ as $[0,1]$  (via the map $S_q=[e^{i\frac{2\pi}{n} (q-1) }, e^{i\frac{2\pi}{n} q }]\rightarrow [0,1], e^{i2\pi x}\mapsto nx-(q-1)$) and let the map
$\varphi_{2k}^j|_{S_q}: S_q\rightarrow S_q$ be isomorphic to
$g_k(x)=x^{t_k}: [0,1]\rightarrow [0,1]$ when $1\le j \le 2^{n_{2k}-1}-1$; and $g_k^{-1}(x)=x^{1/t_k}: [0,1]\rightarrow [0,1] $ when $2^{n_{2k}-1}-1<j\le 2^{n_{2k}}-2$, where $\{t_k\}$ is a decreasing sequence tending to $1.$

To be precise, for each $q\in \{1,2,\ldots,n\}$, when $z=e^{i2\pi x}\in S_q$, $\varphi_{2k}^j: S_q\rightarrow S_q$ is defined as follows:
\begin{equation}\label{C6}
  \varphi_{2k}^j(z)= \varphi_{2k}^j(e^{i2\pi x})=
\left\{
\begin{array}{ll}
\exp\left({i2\pi \frac{(nx-(q-1))^{t_k}+(q-1)}{n}}\right), & 1\le j \le 2^{n_{2k}-1}-1 \hbox{;} \\
\exp\left( {i2\pi  \frac{(nx-(q-1))^{1/t_k}+(q-1)}{n}}\right), & 2^{n_{2k}-1}-1< j\le 2^{n_{2k}}-2 \hbox{.}
\end{array}
\right. \tag{C6}
\end{equation}
where $\{t_k\}$ is a decreasing sequence tending to $1.$
Thus for each $q\in \{1,2,\ldots, n\}$, $\varphi_{2k}^j|_{S_q}: S_q\rightarrow S_q$ has exactly two fixed points, one repulsive and one attractive. Notice that the construction of $\varphi _{2k} ^j$  also satisfies the convention \eqref{C3}, i.e., $\psi _{2k}=id.$



Note that when $n=1,$ $S_1$ is exactly $\bbs^1$ and \eqref{C6} coincides with \eqref{C5}.

\begin{prop}\label{prop-almost-proximal}
Suppose that $\Phi$ satisfies \eqref{C2}, \eqref{C4} and \eqref{C6}. Then the
extension $$\pi: (Z(2)\times \bbs^1, F_\Phi )\to (Z(2),\tau), \ (\a , z)\mapsto \a$$ is an almost proximal extension of uniformly rigid minimal flows, and the cardinality of maximal almost periodic set in each fiber is $n$.
\end{prop}

\begin{proof}
First notice that \eqref{C4} and \eqref{C6} imply \eqref{C1} and \eqref{C3}.  It follows that $(Z(2)\times \bbs^1, F_\Phi )$ is uniformly rigid and minimal by Proposition \ref{prop-min}.
	
Let $\a\in Z(2)$ and $z_0\in \bbs^1$. Let $$x_j=(\a, z_0 e^{i\frac{2\pi}{n}(j-1)})=(\a, z_0\w^{j-1})\in \pi^{-1}\a),\ 1\le j\le n.$$
First we show that $\{x_1, x_2,\ldots, x_n\}$ is an almost periodic set. To attain that aim, we show $(x_1,x_2,\ldots, x_n)$ is a minimal point of $(X^n, F_\Phi^{(n)})$, where $F^{(n)}_\Phi=F_\Phi\times \ldots \times F_\Phi$ ($n$ times).

By the conditions \eqref{C4} and \eqref{C6}, we have
\begin{equation}\label{z1}
  \rho_X(F_\Phi(x_{j}), F_\Phi(x_{j+1}))=\rho_X(x_j, x_{j+1})=\frac{2\pi}{n}, \ 1\le j\le n-1.
\end{equation}
Let $(y_1,y_2,\ldots, y_n)$ be a minimal point of $\overline{\O}((x_1,x_2,\ldots, x_n), F_\Phi^{(n)})$ with $y_1=x_1$. (First we choose any minimal point $(y_1', y_2',\ldots, y_n')$ in $\overline{\O}((x_1,x_2,\ldots, x_n), F_\Phi^{(n)})$. Since $(X,F_\Phi)$ is minimal, there is some $p\in \M$ such that $py_1'=x_1$. Let $(y_1,y_2,\ldots, y_n)=p(y_1', y_2',\ldots, y_n')$. Then $(y_1,y_2,\ldots, y_n)$ is a minimal point of $\overline{\O}((x_1,x_2,\ldots, x_n), F_\Phi^{(n)})$ with $y_1=x_1$.) By \eqref{z1},
$$\rho_X(y_j,y_{j+1})=\rho_X(x_j, x_{j+1})=\frac{2\pi}{n}, \ 1\le j\le n-1.$$
Since $F_\Phi$ preserves orientation, we have
$$[x_1,x_2]=[y_1, y_2]=[x_1,y_2],$$
and hence $x_2=y_2$. By the same reason, we have $x_j=y_j$ for $3\le j\le n$. Thus $(x_1,x_2,\ldots, x_n)=(y_1,y_2,\ldots, y_n)$ is minimal.

Next we show that for each $\a\in Z(2), z_1,z_2\in \bbs^1$, if $\rho_X((\a,z_1),(\a,z_2))<\frac{2\pi}{n}$, then $(\a,z_1),(\a,z_2)$ are proximal. Without loss of generality, we may assume that $\a=\underline{0}$ and $z_1,z_2\in [1,\omega).$
By \eqref{C6}, for  $z=e^{i2\pi x}\in S_1=[1,\omega]$,
\begin{equation*}
  \varphi_{2k}^j(z)= \varphi_{2k}^j(e^{i2\pi x})=
\left\{
\begin{array}{ll}
\exp\left({i2\pi \frac{(nx)^{t_k}}{n}}\right), & 1\leq j \le 2^{n_{2k}-1}-1 \hbox{;} \\
\exp\left( {i2\pi \frac{(nx)^{1/{t_k}}}{n}}\right), & 2^{n_{2k}-1}-1<j\le 2^{n_{2k}}-2 \hbox{.}
\end{array}
\right.
\end{equation*}
Let $m_k$ and $y_m$ be defined as in Lemma \ref{lem-6.4}. According to (\ref{2}), for $1\le s<2^{n_{2k}}$,
	\[F^{s\cdot m_{2k-1}}_\Phi(\underline 0,y_0)=(0^{n_1+n_2+\dots n_{2k-1}}\underline s0^\infty ,y_{s\cdot m_{2k-1}}),\]
	where $y_{s\cdot m_{2k-1}}=\varphi^{s-1}_{2k}\circ \varphi^{s-2}_{2k}\circ \dots \circ \varphi^1_{2k}\circ \varphi ^0_{2k}(y_0).$ Let $z_1=e^{i2\pi a_1}, z_2=e^{i2\pi a_2}$ and $s=2^{n_{2k}-1}$. Then for $j=1,2$,
$$F^{sm_{2k-1}}_\Phi (\underline{0},z_j)=\left(0^{n_1+n_2+\dots n_{2k-1}}\underline s0^\infty, \exp\Big({i2\pi \frac{(na_j)^{t_k^{2^{n_{2k}-1}}}}{n}}\Big)\right)\to (\underline{0}, 1),\  k\to\infty$$
Thus
\[((\underline 0,z_1),(\underline 0,z_2))\in P(Z(2)\times \bbs^1,F_\Phi).\]

To sum up, we have showed that for $\a\in Z(2)$, $A\subseteq \pi^{-1}(\a)$ is an almost periodic set with maximal cardinality if and only if $A=\{(\a, z_0 e^{i\frac{2\pi}{n}(j-1)})=(\a, z_0\w^{j-1}): 1\le j\le n\}$ for some $z_0\in \bbs ^1$. The proof is complete.
\end{proof}

\subsection{Some remarks}\
\medskip

The method to construct the flow $(Z(2)\times \bbs^1, F_\Phi)$ is modified from the examples in \cite{CM06}. This kind of construction originally was from the study of a triangular map of the unit square $[0,1]^2$, which is a continuous map $F: [0,1]^2\rightarrow [0,1]^2$ of the form $F(x,y)=(f(x),g_x(y))$. For a short survey of triangular maps, see \cite{Smi}.

\medskip

We may replace $\bbs^1$ by $\T^n$, ${\bf P}^n$ etc. to get similar minimal flows. Since $\bbs^1$ is enough for our purpose, we do not use them in this paper. But for different purpose, using manifold with higher dimension may be useful.

\section{Further discussion}\label{section-ques}

In this section, we give some questions.

\subsection{Proximality and chaos}\
\medskip

First we restate Problem 5.23. in \cite{AGHSY} as follows.

\begin{ques}\label{ques2}
If a minimal flow  is not point distal ( i.e., for any point $x\in X$ , there is $x'\neq x$ such that $(x,x')$ is proximal), is it chaotic in the sense of Li-Yorke?
\end{ques}

In this paper,we show that if a minimal flow $(X,T)$ is an almost proximal but not almost finite to one extension of some flow $(Y,T)$, then $(X,T)$ is not point distal and it is Li-Yorke chaotic.
See \cite{AGHSY} for another special case about Question \ref{ques2}.


\subsection{Openness and perfectness}\
\medskip

It is a well known fact that for a homomorphism of ergodic measure preserving systems either almost all fibers have constant, finite cardinality or almost all fibers have the cardinality of the continuum.
In Theorem \ref{thm-wm} and Theorem \ref{thm-pr-old}, almost all fibers are perfect.
In fact, in topological case we have the following general result:

\begin{thm}\cite[Theorem 6.31]{AAG}\label{thm-AAG}
Let $\pi: (X,T)\rightarrow (Y,T)$ be an extension of minimal flows. Then one of the following holds:
\begin{enumerate}
  \item $\pi$ is almost finite to one.

  \item Every fiber $\pi^{-1}(y)$ is infinite and $\{y\in Y: \pi^{-1}(y) $ is perfect $\}$ is a residual subset of $Y$.
\end{enumerate}
\end{thm}

But the proof of theorem above can not imply that in (2) every fiber is perfect even for open extensions. Thus we have the following question:

\begin{ques}\label{ques3}
Let $\pi: (X,T)\rightarrow (Y,T)$ be an extension of minimal flows. If $\pi$ is open and is not finite to one, is every fiber $\pi^{-1}(y)$ perfect?
\end{ques}

\subsubsection{Some special cases}\
\medskip

For some special case, we have positive answer for Question \ref{ques3}.
First we have that for some special weakly mixing extensions, each fiber is perfect.

\begin{thm}\cite{SY18}
Let $\pi: (X,T)\rightarrow (Y,T)$ be a non-trivial weakly mxing RIC extension of minimal flows. Then each fiber is perfect.
\end{thm}

In the rest of the section, we show that if a distal extension is not finite to one, each fiber is perfect. To prove this result, we need Furstenberg-Ellis's structure theorem.

Furstenberg's structure theorem for distal flows \cite{F63} says that any distal minimal flow can be constructed by equicontinuous extensions. We state the result in its relative version by Ellis \cite{Ellis}. Let $\pi: (X,T)\rightarrow (Y,T)$ be a distal extension of minimal flows. Then there is an ordinal $\eta$
(which is countable when $X$ is metrizable) and a family of flows
$\{(Z_n,T)\}_{n\le\eta}$ such that
\begin{enumerate}
  \item[(i)] $Z_0=Y$,
  \item[(ii)] for every $n <\eta$ there exists a homomorphism
$\rho_{n+1} :Z_{n+1}\to Z_{n}$ which is equicontinuous,
  \item[(iii)] for a limit ordinal $\nu\le\eta$ the flow $Z_\nu$
is the inverse limit of the flows $\{Z_\iota\}_{\iota<\nu}$,
\item[(iv)] $Z_\eta=X$.
\end{enumerate}
\begin{equation}\label{a5}
  Y= Z_0  \stackrel{\rho_1} \longleftarrow  Z_1   \stackrel{\rho_2}\longleftarrow \cdots  \stackrel{\rho_n} \longleftarrow Z_n  \stackrel{\rho_{n+1}} \longleftarrow Z_{n+1}
 \longleftarrow \cdots \stackrel{\rho_{\eta}} \longleftarrow Z_\eta=X.
\end{equation}

When $Y=\{pt\}$ is the trivial flow, we have the structure theorem for a distal minimal flow.

\begin{thm}
Let $\pi: (X,T)\rightarrow (Y,T)$ be a distal extension of minimal flows. If $\pi$ is not finite to one, then each fiber is perfect.
\end{thm}

\begin{proof}
We need an equivalent characterization of an equicontinuous extension.
Let $M$ be a homogeneous compact metric space. By this we mean a compact metric space such that for any two points $x,y\in M$, there is an isometry of $M$ taking $x$ into $y$. The isometries of $M$ form a compact group $H$, $M$ may be identified with a coset space $H/H_0$, where $H_0$ is the subgroup of $H$ leaving a given point of $M$ fixed.

Let $\pi: X\rightarrow Y$ be an extension of flows. $\pi$ is equicontinuous if and only if there exists a continuous map $\rho: R_\pi\rightarrow \R$ such that for each $y\in Y$, $\rho$ defines a metric on the fiber $X_y=\pi^{-1}(y)$ under which $X_y$ is isometric to $M$, and $\rho(tx_1,tx_2)=\rho(x_1,x_2)$ for all $t\in T$ and $(x_1,x_2)\in R_\pi$ \cite{F63}.
Thus if $\pi$ is equicontinuous, then each fiber is isometric to $M=H/H_0$. Hence if $\pi$ is not finite to one, then each fiber is perfect.

To deal with distal extensions, we use Furstenberg-Ellis structure theorem as stated above. Let $\{(Z_n,T)\}_{n\le\eta}$ be the factors. Since $\pi$ is not finite to one, either $\eta$ is not finite ordinal and each $\rho_n$ is finite to one, or there is some $n\le \eta$ such that $\rho_n$ is not finite to one. In the first case, each fiber is an inverse limit of finite sets and it is a Cantor set; in the second case, $\rho_n$ is infinite to one equicontinuous extension and each fiber of $\rho_n$ is perfect and by this we claim that each fiber of $\pi$ is also perfect. To prove the second case, we need the following claim: if $\pi_1:(X_1,T)\to (X_2,T) ,$ $\pi_2:(X_2,T)\to (X_3,T)$ are open extensions such that each fiber of $\pi_1$ or $\pi_2$ is perfect, then each fiber of $\pi_2 \circ \pi_1$ is perfect. First by definition it is easy to see that when each fiber of $\pi_1$ is perfect, we have each fiber of $\pi_2 \circ \pi_1$ is perfect. Next we show the other case. Suppose that each fiber of $\pi_2$ is perfect. We show that for each $z\in X_3$, $(\pi_2 \circ \pi_1)^{-1} (z)$ is perfect. Let $x\in (\pi_2 \circ \pi_1)^{-1} (z)$ and  $y=\pi_1(x).$ Clearly $y\in \pi_2^{-1}(z),$ and by perfectness of $\pi_2^{-1}(z)$ one can find $y_n \in \pi_2^{-1}(z)$ such that $y_n \to y$ as $n\to\infty.$ Since $\pi_1$ is open, there exists $x_n\in X$ such that $\pi_1(x_n)=y_n$ and $x_n \to x$ as $n\to\infty$. To sum up, there are $x_n \in (\pi_2 \circ \pi_1)^{-1} (z)$ such that $x_n \to x$, $n\to\infty.$ Thus each point of $(\pi_2 \circ \pi_1)^{-1} (z)$ is not isolated and $(\pi_2 \circ \pi_1)^{-1} (z)$ is perfect. Thus we have the claim. By this claim and Furstenberg-Ellis structure theorem \eqref{a5}, one can show that each fiber of $\pi$ is perfect.
\end{proof}

\subsection{Open proximality and Entropy}\
\medskip

Our last question is about entropy. Let $\pi: (X,\Z)\rightarrow (Y,\Z)$ be an extension of discrete flows. If $h_{\rm top}(X)> h_{\rm top} (Y)$, then lots of fiber will have very complex properties (see \cite{Zhang} for example). We are not sure that open proximal extensions can reach those kinds of complexity. Thus we have the following question.

\begin{ques}\label{ques4}
Let $\pi: (X,\Z)\rightarrow (Y,\Z)$ be an open proximal extension of discrete minimal flows. Is it true that $h_{\rm top}(X)=h_{\rm top} (Y)$?
\end{ques}

Note that in Question \ref{ques4}, openness is a necessary condition, since there are lots of minimal flows which are almost one to one extensions of their maximal equicontinuous factors, and they have positive entropy \cite{FW89}.


\end{document}